\documentclass[11pt]{article}
\usepackage[utf8]{inputenc}
\usepackage{amsfonts, amsmath}
\usepackage{amssymb}
\usepackage{amsthm,amsmath,amscd}
\usepackage{enumerate}
\usepackage{url}
\usepackage{amsbsy}
\usepackage[pdftex]{graphicx}
\usepackage[export]{adjustbox}
\usepackage[margin=25pt,font=small,labelfont=bf]{caption}
\usepackage[breaklinks,hyperfootnotes=false,bookmarks=false,colorlinks=true]{hyperref}
\usepackage{subfig}
\usepackage{mdframed}
\usepackage{color}
\usepackage[square,numbers,sort&compress]{natbib}
\newcommand*{\doi}[1]{doi: \href{https://dx.doi.org/#1}{\urlstyle{rm}\nolinkurl{#1}}}
\newcommand*{\arxiv}[1]{arXiv:  \href{https://arxiv.org/abs/#1}{\urlstyle{rm}\nolinkurl{#1}}}

\newcommand{\defn}[1]{\textcolor{blue}{\emph{#1}}}

\graphicspath{{figs/}}

\def\qqed{\hfill {\hfill $\Box \Box$} \medskip}

\DeclareMathOperator{\sing}{Sing}

\newcommand{\RR}{\mathbb R}

\newcommand{\EE}{\mathbb E}
\newcommand{\R}{\mathbb R}

\newcommand{\bna}{\begin{eqnarray}}
\newcommand{\ena}{\end{eqnarray}}
\newcommand{\ba}{\begin{eqnarray*}}
\newcommand{\ea}{\end{eqnarray*}}
\newcommand{\bs}[1]{}

\newtheorem{theorem}{Theorem}[section]

\newtheorem{lemma}[theorem]{Lemma}
\newtheorem{proposition}[theorem]{Proposition}
\newtheorem{definition}[theorem]{Definition}
\newtheorem{remark}[theorem]{Remark}

\newtheorem{question}[theorem]{Question}

\DeclareMathOperator{\Dim}{Dim}

\DeclareMathOperator{\tr}{tr}

\newcommand{\CC}{{\mathbb C}}

\newcommand{\CT}{{\mathcal A}}

\newcommand{\CL}{{\mathcal L}}
\newcommand{\QQ}{{\mathbb Q}}
\newcommand{\ZZ}{{\mathbb Z}}

\def\p{{\bf p}}

\def\pn{\p =(\p_1, \dots, \p_{n}) }
\def\q{{\bf q}}

\def\v{{\bf v}}

\def\e{{\bf e}}

\def\Q{{\bf Q}}

\def\x{{\bf x}}
\def\y{{\bf y}}

\def\t{{\bf t}}

\def\M{{\bf M}}

\DeclareMathOperator{\lin}{lin}

\title{Generic Unlabeled Global Rigidity}

\author{
Steven J. Gortler\thanks{Partially supported by NSF grant DMS-1564473}
\and 
Louis Theran\and 
Dylan P. Thurston}

\date{}

\usepackage[margin=1in]{geometry}
\begin{document}
\maketitle 
\begin{abstract}
Let $\p$ be a configuration of $n$ points in $\RR^d$ for some 
$n$ and some
$d \ge 2$. Each pair of points has a Euclidean
distance in the configuration. Given some 
graph $G$ on $n$ vertices, 
we measure the point-pair distances corresponding
to the edges of $G$.

In this paper, 
we study the question of when a generic
$\p$ in $d$ dimensions
will be uniquely determined (up to an
unknowable Euclidean transformation) 
from a given set of point-pair distances together
with knowledge of $d$ and $n$.
In this setting the 
distances are given simply as a set of real numbers; they are 
not labeled with the combinatorial data that describes 
which 
point-pair 
gave rise to which distance, nor is data about
$G$ given.

We show,
perhaps surprisingly, 
that in terms of generic uniqueness, labels have
no effect. A generic configuration is 
determined by an unlabeled set of point-pair
distances
(together with $d$ and $n$) iff it is determined by the 
labeled distances.
\end{abstract}

\section{Introduction}

Let $d$ be some fixed dimension.
\begin{definition}
An \defn{ordered graph} $G = (V,E)$ on $n$ vertices 
$V = \{1,\ldots,n\}$ is an ordered\footnote{The ordering is 
just for notational convenience.} sequence of edges
(unordered vertex pairs). We do not allow self-loops or
duplicate edges.
\end{definition}
Let $G$ be an ordered graph 
(with $n\ge d+2$ vertices and $m$ edges)
and   $\pn$ be a configuration of $n$ points in $\RR^d$,
which we associate with the vertices of $G$ in the 
natural way.
One can
measure the squared Euclidean distances in $\RR^d$
between vertex pairs corresponding to the edges of  $G$. 
This gives us an ordered sequence, $\v$,
of $m$ squared-distance real values. We write this as 
$\v=m^{\EE}_G(\p)$, where $m^{\EE}_G(\cdot)$ maps from configurations
to squared edge-lengths along the edges of $G$ (the $\EE$ superscript
denotes Euclidean).
Importantly, $\v$ does not contain any labeling information
describing which squared-length value is associated to
which vertex-pair; it is simply a sequence of real numbers.

A natural question is:
\begin{center}
\textit{When does $\v$ 
(together with $d$ and $n$)
determine $G$ and $\p$?} 
\end{center}
We can only hope
for $G$ to be unique up to a relabeling of its vertices.
A relabeling is simply a permutation on the vertices,
$\{1,\ldots,n\}$.
Moreover, under this relabeling, we can only hope that
that $\p$
is unique up to a congruence (affine isometry) of $\RR^d$.
Thus, given some other configuration $\q$ and ordered graph
$H$, also with $n$ vertices and $m$ edges, 
such that $\v=m^{\EE}_H(\q)$, under what conditions will we know
that 
$G=H$ up to a vertex 
relabeling and
$\p=\q$ up to congruence?
The restriction that $H$ has exactly $n$ 
vertices is  natural; if $H$ were, say, a tree over $m+1$ 
vertices, it  would be able to produce any $m$-tuple of real numbers including
$\v$ as the squared distance measurement of some configuration.

We will be interested in studying this problem under the non-degeneracy 
assumption that $\p$ is generic.
\begin{definition}
\label{def:genConig}
A configuration $\p$ in $\RR^d$
is \defn{generic} if there
is no non-zero  polynomial relation, with coefficients in $\QQ$,
among the coordinates of $\p$.
\end{definition}

Boutin and Kemper~\cite{BK1} proved that if $G$ 
consists of an ordering of the edges of the
complete graph, $K_n$, and $\p$ is generic, then
uniqueness is guaranteed. There is only one $\p$, up to
a congruence, 
consistent with its unlabeled $\v$.
With this result in hand, 
one can immediately weaken the completeness 
requirement for $G$, and only require that it 
``allows for trilateration'' in $d$ dimensions.
Loosely speaking, this means that $G$ can be built by
gluing together overlapping $K_{d+2}$ graphs
(see~\cite{loops} for formal definitions). This
unlabeled trilateration concept was
first explored in~\cite{dux1}, and a formal proof of 
uniqueness  is given in~\cite{loops}.

Our goal in this paper is to  weaken the conditions on $G$ as much as possible. 
\begin{definition}\label{def:global-rigidity}
Let 
$G$
be an ordered graph and $\p$  a configuration
in $\RR^d$.
We say that that the pair $(G,\p)$ 
is \defn{globally rigid} in $\RR^d$  if
for all configurations $\q$ in $\RR^d$,
$m^{\EE}_G(\p) = m^{\EE}_G(\q)$ implies 
$\p=\q$ (up to congruence).

We say that $G$ 
is \defn{generically globally rigid} in $\RR^d$  if $(G,\p)$ is globally rigid 
for all generic $\p$ in $\RR^d$.
\end{definition}
Gortler, Healy and Thurston \cite{ght} proved:
\begin{theorem}[\cite{ght}]\label{thm:ght}
If an ordered graph $G$ is not generically globally rigid 
in $\RR^d$, then for any generic $\p$, there is a 
non-congruent $\q$ so that $m^{\EE}_G(\p) = m^{\EE}_G(\q)$.
\end{theorem}
This means, in particular, that every graph is either 
generically globally rigid or generically not globally 
rigid.

Ordered graphs 
that allow for $d$-dimensional trilateration are  
generically globally rigid
in $\RR^d$ (see, e.g., \cite{GGLT13}), 
but there are many\footnote{For $d=2$ and $G$ with $m = O(n\log n)$ edges, results from \cite{JSS07,KMT11}
imply that almost all globally rigid graphs do 
not allow for trilateration.}
graphs that are generically
globally rigid but do not allow for trilateration. A small 
example in two dimensions
is when $G$ comprises the edges of the complete bipartite graph $K_{4,3}$ 
(generic global rigidity follows from the combinatorial considerations of~\cite{conGR,jj}
and can be directly confirmed
using the algorithm from~\cite{conGR,ght}). 
This graph 
does not even contain a single triangle!

If an ordered graph 
$G$ is not generically globally rigid, then one generally
cannot recover $\p$ when given 
\emph{both} $\v$ and $G$ (that is, labeled data).
The recovery problem is 
simply not well-posed. When an ordered graph is generically globally rigid, then  
generally this labeled 
recovery problem will be well-posed, though it still might be
intractable to perform~\cite{saxe}. 
We note that testing whether an 
ordered graph  is generically
globally rigid can be done with an efficient randomized algorithm~\cite{ght}.

From the above, it is clear that generic global
rigidity is necessary for generic unlabeled uniqueness. 
In this paper we prove the following theorem which states
that the property of generic global rigidity of a graph
is also sufficient for generic unlabeled uniqueness.
This result answers a question 
posed in~\cite{loops}.
\begin{theorem}
\label{thm:main}
In any fixed dimension $d \ge 2$, 
let $\p$ be a generic configuration of $n \ge d+2$ 
points. Let
$\v= m^{\EE}_G(\p)$, where 
$G$ 
is an ordered graph 
(with $n$ vertices and $m$ edges)
that is generically 
globally rigid in $\RR^d$.

Suppose there is a 
configuration $\q$, 
also of $n$ points,
along with 
an ordered graph $H$
(with $n$ vertices and $m$ edges)
such that 
$\v= m^{\EE}_H(\q)$.

Then
there is a vertex relabeling of $H$ such that
$G=H$.
Moreover, under this vertex relabeling,
up to congruence,
$\q=\p$.
\end{theorem}
\begin{remark}
This theorem is true in one dimension as well, if we add the assumption that $G$ is $3$-connected.  (This assumption will come for free in higher dimension.)
We will, in fact,
use $3$-connectivity in the proof of the more technical Theorem 
\ref{thm:iso} that underlies our main result.
\end{remark}
\begin{remark}
We can state Theorem \ref{thm:main} without ordered graphs as 
follows.  Let $G$
be an unordered generically globally rigid graph in dimension $d\ge 2$, 
and $\p$ be a generic configuration in dimension $d$.  If $H$ is some other
unordered graph with the same number of vertices as $G$ and $\q$ any configuration 
so that $(H,\q)$ has the same unordered set of edge lengths as $(G,\p)$, 
Theorem \ref{thm:main} implies that there is an isomorphism between $G$ and 
$H$ consistent with the bijection on edges induced by the distinct edge lengths of 
a generic measurement.  Furthermore, under 
this isomorphism, $\p$ is congruent to $\q$.

Ordered graphs are a convenience  to avoid referring to 
an implicit isomorphism throughout.
\end{remark}

\begin{remark}
Theorem \ref{thm:ght} implies that 
a generic configuration $\p$ is determined
by its labeled edge lengths iff these edges form a generically globally rigid
graph.  Hence, Theorem \ref{thm:main} says that
 a generic configuration $\p$, 
with known $d$ and $n$,
is uniquely determined (up to relabeling 
and congruence) 
from its (unordered)
unlabeled edge lengths 
iff it is uniquely 
determined (up to congruence)
by its labeled edge lengths.
\end{remark}

Note that for a generically globally rigid graph $G$, 
there can be a non-generic
$(G,\p)$ which is still globally rigid, but for which
$m_G(\p)$ (and $n$)  does not uniquely determine $\p$
in the unlabeled setting. (See~\cite[Figure 4]{BK1}
for an example in the plane where $G$ is $K_4$.)

Since the non-generic failures of this theorem are due to a finite
collection of algebraically expressible exceptions, the uniqueness promised by this theorem actually holds over a 
Zariski open set of configurations.

Our result is information theoretic; it does not give an efficient
algorithm for determining $\p$ from  $\v$.
Indeed, determining $\p$ is
NP-hard, even when given $\v$ \textit{and} $G$~\cite{saxe}.
We will discuss some practical implications and related
questions in Section \ref{sec: practical}.

The body of this paper will be concerned with the proof
of Theorem \ref{thm:main}.  Our approach is to 
reduce the question to one about  the
so-called ``measurement variety'' (defined
in Section \ref{sec:mv}) of $G$, which represents all possible $\v$, as $\p$ varies over all $d$-dimensional configurations. We will want to understand when two distinct ordered graphs, $G$ and $H$, can give rise to the same measurement variety. We will find
(see Theorem \ref{thm:iso}, below)
that when $G$ is generically globally rigid in $d$ dimensions,
then this cannot happen. Theorem \ref{thm:main} then 
follows quickly.

\subsection*{Acknowledgements}
We thank Brian Osserman for fielding algebraic geometry
queries, Meera Sitharam for feedback and discussions 
on the relationship to matrix completion, and Robert 
Krone for an interesting question about coincident points.

\section{Rigidity Background}

In this section we will recall the needed definitions and results
from graph rigidity theory.

\subsection{Local Rigidity}

\begin{definition}
A \defn{framework} $(G,\p)$ is an pair of an ordered graph
and a configuration.  Two frameworks $(G,\p)$ and $(G,\q)$ 
are \defn{equivalent} if $m^\EE_G(\p) = m^\EE_G(\q)$; 
they are \defn{congruent} if $\p$ and $\q$ are congruent.
\end{definition}

\begin{definition}\label{def:local-rigidity}
Let 
$G$
be an ordered graph. 
We say that $(G,\p)$ 
is \defn{locally rigid} in $\RR^d$ if,
a sufficiently small enough neighborhood of $\p$ in the fiber 
$(m^{\EE}_G)^{-1}(m^{\EE}_G(\p))$ consists only of $\q$ that are congruent 
to $\p$. Otherwise we say that $(G,\p)$ is \defn{locally flexible} in $\RR^d$.
\end{definition}
The fiber of $m^\EE_G$ consists of the configurations $\q$ such that $(G,\q)$ is equivalent to 
$(G,\p)$.  So local rigidity means that there is a neighborhood of $\p$ 
in which any $\q$ with $(G,\q)$ equivalent to $(G,\p)$ must be 
congruent to $\p$, in parallel to Definition \ref{def:global-rigidity}.

\begin{definition}
A \defn{first-order flex} or \defn{infinitesimal flex} $\p'$ 
in $\RR^d$ of $(G,\p)$ 
is a
corresponding assignment of vectors $\p'=(\p_1', \dots, \p_n')$,
$\p'_i \in \RR^d$ 
such
that for each $\{i,j\}$, an edge of $G$, the following
holds:
\begin{eqnarray}
(\p_i - \p_j)\cdot (\p'_i - \p'_j)&=&0 \label{eqn:first} 
\end{eqnarray}

A first-order flex $\p'$ in $\RR^d$
is \defn{trivial} if it is the restriction to the vertices of the
time-zero derivative of a smooth
motion of isometries of $\R^d$. 
\end{definition}
The property of being trivial is independent of the graph $G$.

\begin{definition}
A framework $(G,\p)$ in $\RR^d$ 
is called \defn{infinitesimally rigid} in $\R^d$
if it has no infinitesimal flexes in $\R^d$ except for trivial ones.
When $n \ge (d+1)$ this is the same as saying that the rank of 
the differential of $m^{\EE}_G(\cdot)$ at $\p$ is 
$nd-\binom{d+1}{2}$.
If a framework is not infinitesimally rigid in $\RR^d$, 
it is called \defn{infinitesimally flexible} in $\RR^d$.
\end{definition}
We need some standard facts about infinitesimal rigidity.

\begin{theorem}[See e.g., \cite{gluck}]
\label{thm:ir-lr}
If $(G,\p)$ is infinitesimally rigid in $\RR^d$, 
then $(G,\p)$ is locally rigid in $\RR^d$. 
\end{theorem}
Affine transformations $A$ on $\RR^d$ act on configurations 
pointwise to produce another configuration, i.e., $A(\p)_i := A(\p_i)$.
\begin{lemma}[\cite{CW82}]
\label{lem:ir-aff}
Let $(G,\p)$ be a framework in $\RR^d$ and let 
$A$ be a non-singular affine transformation.  Then
$(G,\p)$ is infinitesimally rigid if and only
if $(G,A(\p))$ is.
\end{lemma}
In other words, infinitesimal rigidity is invariant under
affine transformations.
The following two statements are folklore, but we give 
proofs for completeness.
\begin{lemma}\label{lem:not-flat}
Let $G$ be a graph with $n\ge d+1$ vertices and let
$(G,\p)$ be an infinitesimally rigid framework in $\RR^d$.
Then $\p$ has $d$-dimensional affine span.
\end{lemma}
\begin{proof}
Any assignment of vectors orthogonal to the affine span of $\p$
is an infinitesimal flex of $(G,\p)$.
Hence, if $\p$ has defective affine span, there is, at least, an 
$n$-dimensional 
space of infinitesimal flexes of $(G,\p)$ orthogonal to the 
affine span of $\p$.  There is also, at least,
a $\binom{d}{2}$-dimensional space (from rigid motions in dimension $d-1$) 
of infinitesimal flexes within the affine span of $(G,\p)$.  Thus 
$(G,\p)$ has infinitesimal flex space of dimension at least
$\binom{d}{2} + d + 1 > \binom{d+1}{2}$.
\end{proof}

\begin{lemma}
\label{lem:ccomp}
Let $(G,\p)$ be a framework. Then, up to congruence, there are only
a finite number of configurations $\q$ so that $(G,\q)$ is locally rigid
and equivalent to $(G,\p)$.
\end{lemma}
\begin{proof}
The set of frameworks that are equivalent to $\p$ form an algebraic
variety $V$. 
From the definition of local rigidity, if $\q$ is in $V$
and locally rigid, then it is only connected in $V$ to other frameworks
in its congruence class (in fact only ones that do not involve reflection).
Thus an infinite number of such $\q$ would imply an infinite number
of connected components in $V$.
But
as a variety,  $V$ must have a finite
number of connected components.
\end{proof}

\begin{definition}
If $(G,\p)$ is locally rigid for all generic configurations
$\p$ in $\RR^d$, then we say that $G$ is \defn{generically locally rigid} in $\RR^d$. 
If $(G,\p)$ is locally flexible for all generic configurations
$\p$ in $\RR^d$, then we say that $G$ is \defn{generically 
locally flexible} in $\RR^d$. 

If $(G,\p)$ is infinitesimally rigid for all generic configurations
$\p$ in $\RR^d$, then we say that $G$ is \defn{generically 
infinitesimally rigid} in $\RR^d$. 
If $(G,\p)$ is infinitesimally flexible for all generic configurations
$\p$ in $\RR^d$, then we say that $G$ is \defn{generically 
infinitesimally flexible} in $\RR^d$. 
\end{definition}

As described in~\cite{asimow},
generic local rigidity is determined by generic infinitesimal 
rigidity.
\begin{theorem}[\cite{asimow}]
\label{thm:glr}
If some framework $(G,\p)$ in $\RR^d$
is infinitesimally rigid in $\RR^d$, then 
$G$ is generically infinitesimally rigid in $\RR^d$ and thus
generically 
locally rigid in $\RR^d$.
If  $G$ is not
generically infinitesimally rigid in $\RR^d$ then it is 
generically locally flexible in $\RR^d$.
Thus, if $G$ is not generically locally rigid in $\RR^d$
then it is 
generically locally  flexible in $\RR^d$.
\end{theorem}

\subsection{Global Rigidity}

The following two results about generic global rigidity
will be useful.
\begin{lemma}
\label{lem:descend}
Let $G$ be generically globally rigid in $\RR^d$. Then $G$ 
is generically globally rigid in $\RR^{d-1}$.
\end{lemma}
\begin{proof}
If $G$ is generically globally rigid in dimension $d$, then
it remains so under \emph{coning}, 
the process of adding one vertex
and attaching it to all vertices in $G$. 
A result of Connelly and Whiteley,  
\cite[Corollary 10]{cone}, then implies
that $G$ is generically globally 
rigid in $\RR^{d-1}$.
\end{proof}

A theorem of Hendrickson relates generic 
global rigidity and connectivity:
\begin{theorem}[\cite{hen}]
\label{thm:hen}
Let $G$ be generically globally rigid in $\RR^d$. Then $G$ 
is $d+1$-connected.
\end{theorem}

Now we review idea of equilibrium stresses and how they
relate to global rigidity.

\begin{definition}
\label{def:stress}
Given an ordered  graph $G$,
 a \defn{stress vector} $\omega =( \dots,
 \omega_{ij}, \dots )$, is an assignment of a real scalar
 $\omega_{ij}=\omega_{ji}$ to each edge, $\{i,j\}$ in $G$.  (We have
 $\omega_{ij}=0$ when $\{i,j\}$ is not an edge of $G$.)

We say that 
 $\omega$ is an \defn{equilibrium stress vector}
for $(G,\p)$ if the vector equation
\begin{equation}\label{eqn:equilibrium}
 \sum_j \omega_{ij}(\p_i-\p_j) = 0
 \end{equation}
holds for all vertices $i$ of $G$.

We associate an $n$-by-$n$
\defn{stress matrix} $\Omega$ to a stress vector $\omega$, 
by setting the $i,j$th entry of $\Omega$ to
$-\omega_{ij}$, for $i \ne j$, and the diagonal entries of $\Omega$
are set such that the row and column sums of $\Omega$ are zero.  
 
If $\omega$ is an equilibrium stress vector for
$(G,\p)$ then we say that the associated $\Omega$ is an 
\defn{ equilibrium stress matrix}
 for $(G,\p)$. For each of the  $d$ spatial dimensions,
if we define a vector $\v$ in $\RR^n$ by collecting  the  associated 
coordinate over all of the points in $\p$, we have $\Omega \v=0$.
The all-ones vector is also in the kernel of $\Omega$.
Thus if the 
dimension of the affine span of the vertices $\p$ is $d$, then the
rank of $\Omega$ is at most $n-d-1$, but it could be less.
\end{definition}

\begin{definition}\label{def:stress-kernel}
Let $S$ be linear space of stress matrices.  We 
define the \defn{shared stress kernel} of $S$ to be the
subspace of $\RR^n$ consisting of vectors in the kernel
of every $\Omega\in S$.

The shared stress kernel of a framework 
$(G,\p)$ is the shared stress kernel of the linear space of
equilibrium stress matrices for $(G,\p)$.

From  
the equilibrium condition, we see that the shared
stress  kernel of $(G,\p)$ 
contains the $d$ coordinates of $\p$ along with the 
all-ones vector.  Thus, if the 
dimension of the affine span of the vertices $\p$ is $d$, then the
dimension of the shared stress kernel
is at least $d+1$, but it could be more.
\end{definition}

Below is the central theorem we shall use that connects
generic global rigidity with the dimension of the shared
stress kernel at generic $\p$.

\begin{theorem}[{\cite[Theorems 1.14 and 4.4]{ght}}]
\label{thm:sharedKernel}
Let $G$ be an ordered graph with $n \ge d+2$ vertices.
If $G$ is generically globally rigid
in $\RR^d$, then there is a generic $\p$ with  an equilibrium stress matrix of 
rank $n-d-1$. 
Thus there is a generic $\p$ with a 
shared stress kernel of dimension $d+1$.

If $G$ is not generically globally rigid
in $\RR^d$, then 
every generic $\p$ has
shared stress kernel of dimension $> d+1$. (This direction
is essentially Connelly's sufficient condition~\cite{conGR}
as strengthened slightly in~\cite[Section 4.2]{ght}).
\end{theorem}
\begin{remark}
\label{rem:shared}
From general principles about matrices and rank, 
if one generic framework  has 
an equilibrium stress matrix of rank $n-d-1$, then so
too must all generic frameworks
(see~\cite[Theorem 2.5]{henMol}
and~\cite[Lemma 5.8]{ght}). This also implies that
every complex generic framework also must have 
an equilibrium stress matrix of rank $n-d-1$.
\end{remark}

\section{Measurement Variety}
\label{sec:mv}

In this section, we define the measurement variety and
reduce Theorem~\ref{thm:main} to a statement about measurement varieties. 

From here on out, (unless where explicitly stated)
we move the complex setting, where $\p$ is a configuration of 
$n$ points in~$\CC^d$. This will allow us to apply basic machinery from algebraic geometry to 
our problem. Unless stated otherwise, we will always be
dealing with the Zariski topology, 
where the closed sets are
the algebraic subsets, and Zariski open subsets are obtained by
removing a subvariety from a variety.

\begin{definition}
Let $d$ be some fixed dimension and $n$ a number of vertices.
Let $G:= \{E_1,\ldots, E_m\}$ be  an ordered graph.
The ordering on the edges of $G$
fixes  an association between each edge in $G$ 
and a coordinate axis of $\CC^{m}$. 
Let $m_G(\p)$ be the map from $d$-dimensional 
configuration space to $\CC^{m}$
measuring the squared lengths of the edges
of $G$. 

The \defn{complex squared length} on edge $ij$ is 
\ba
m_{ij}(\p) := \sum_{k=1}^{d}(\p^k_i-\p^k_j)^2
\ea
where $k$ indexes over the $d$ dimension-coordinates. Here, we measure
complex squared length using the complex square operation with no
conjugation.

We denote by $M_{d,G}$ the 
closure of 
the image of 
of $m_G(\cdot)$ over all $d$-dimensional configurations.
This is an algebraic set,
defined over $\QQ$.
We call this the (squared) \defn{measurement variety}
of $G$ (in $d$ dimensions).

As the closure of the image of an irreducible set 
(configuration space), 
under a 
polynomial map, the variety $M_{d,G}$
is irreducible. As $M_{d,G}$ contains all scales of all of its points,
the variety is homogeneous.
\end{definition}

In the complex setting, using the above definition for complex squared 
length, 
we can also define the concepts
of congruence and
infinitesimal/local/global rigidity in $\CC^d$.
Importantly, as described in the following result, 
moving to the complex setting will maintain the rigidity properties relevant to us.  Thus, 
we may simply talk about ``rigidity 
in $d$ dimensions'', without specifying $\RR^d$ or $\CC^d$.

\begin{theorem}
A graph $G$ is generically  infinitesimally/locally/globally rigid 
in $\RR^d$ iff it is so in $\CC^d$.
\end{theorem}
The case of generic global rigidity is proven in~\cite{cggr}.  
One 
direction of generic local rigidity is in \cite{ST10}.  
For 
completeness, here we sketch a proof of the equivalence for generic 
infinitesimal and generic local rigidity.
\begin{proof}
First, we note that a generic real
configuration in $\RR^d$ is also generic as a 
complex configuration.  

Secondly, the proof of Theorem \ref{thm:glr} in~\cite{asimow}, which equates
generic infinitesimal rigidity to generic local rigidity, 
equally applies to the complex setting.

Finally, the rank of the rigidity matrix does not change when enlarging the
field from $\RR$ to $\CC$
(because the determinant is defined over $\ZZ$), and so infinitesimal 
rigidity of a real generic
$(G,\p)$ will be the same in both fields.  By the complex version
of Theorem~\ref{thm:glr}, generic local rigidity as well.
\end{proof}

\begin{lemma}
\label{lem:infrig}
If $G$ is generically locally rigid in $\CC^d$, 
then the image of $m_G(\cdot)$ acting on all configurations
is $dn-\binom{d+1}{2}$-dimensional. Otherwise, 
the dimension of the image is smaller.
\end{lemma}
\begin{proof}[Proof sketch]
From Theorem~\ref{thm:glr}, 
if $G$ is generically locally rigid in $\CC^d$ then
it is generically infinitesimally rigid in $\CC^d$.  Thus 
the generic rank 
of the differential of $m_G(\cdot)$ is $dn-\binom{d+1}{2}$.
From 
the constant rank theorem (as used in~\cite[Proposition 2]{asimow}),
the dimension of the image of $m_G(\cdot)$ is at
least as big as the rank $r$ of the  differential at a generic $\p$.
This is the largest differential rank of $m_G(\cdot)$ over the domain.
Applying Sard's Theorem  to $m_G(\cdot)$ 
(once the non-smooth points of the image are removed, and then the preimages
of these non-smooth points are removed from the domain)
tells us that inverse image of some
(in fact, almost every) point
in the image consists entirely
of configurations $\p$, where the differential
has rank at least as big as the
the dimension of the image of $m_G(\cdot)$.
\end{proof}

The main theorem about measurement varieties we will prove in 
this paper is the following:

\begin{theorem}
\label{thm:iso}
Suppose that $d \ge 2$ (or suppose that  
$d = 1$ and G is $3$-connected).
Let $G$ and $H$ be 
ordered graphs, both with $n \ge d+2$ vertices
and $m$ edges.
Suppose 
$G$ is generically globally rigid in $d$ dimensions. 
Suppose $M_{d,G} =M_{d,H}$.
Then
there is a vertex relabeling under which $G=H$.
\end{theorem}


Assuming
Theorem \ref{thm:iso}, we are ready to prove our main 
result.

\begin{proof}[Proof of Theorem~\ref{thm:main}]
Lemma \ref{lem:infrig} implies that 
$M_{d,G}$ is an irreducible variety of dimension
$dn-\binom{d+1}{2}$.
Meanwhile, using Lemma \ref{lem:infrig} again,
$M_{d,H}$ is an irreducible variety of dimension 
$\le  dn-\binom{d+1}{2}$,
with equality if $H$ is generically locally rigid in $\CC^d$.
(It is here where we need that $H$ does not have more vertices
than $G$.)
The generic real configuration $\p$ is 
also generic
as a point in $\CC^{dn}$.
The point $\v\in \CC^m$ is, by assumption, 
in both $M_{d,G}$ and
$M_{d,H}$ and by Lemma~\ref{lem:genMap},
$\v$ is generic in $M_{d,G}$.
This implies that we must have  
$M_{d,G}  \subseteq M_{d,H}$, 
otherwise  $\v$ would be cut out from $M_{d,G}$
by the one of the equations defining $M_{d,H}$, and thus 
rendering $\v$ non-generic in $M_{d,G}$.
So $M_{d,H}$ must be of dimension at least
$dn-\binom{d+1}{2}$, and thus exactly
$dn-\binom{d+1}{2}$.

Since $M_{d,G}$ and $M_{d,H}$ have the same 
dimension and $M_{d,H}$ is irreducible, 
$M_{d,G}  \subseteq M_{d,H}$ implies that 
$M_{d,G} = M_{d,H}$.

Now we may apply 
Theorem~\ref{thm:iso} to conclude that there 
there is a vertex relabeling
such that $G=H$. Finally, from the assumption that
$G$ is generically globally rigid, we must have
$\p$ congruent to $\q$.
\end{proof}

With this settled, the next two sections 
develop the proof of Theorem~\ref{thm:iso}.
Briefly, the approach is by induction on 
dimension. This kind of induction 
was used in~\cite{loops} to obtain a new
proof of the result of Boutin and Kemper on complete graphs.
The base case, $d=1$, follows 
from a graph-theoretic result of Whitney
via a connection between cycle spaces of graphs and 
projections of $1$-dimensional measurement sets. 
This is done in Section \ref{sec:d=1}.  
The connection
between measurement varieties and Whitney's theorem was 
first explored
in the unpublished manuscript~\cite{miso}. The main results 
from~\cite{miso} are included in Section~\ref{sec:bonus} below.
The 
more difficult step is the inductive one,
which requires understanding the geometry of 
the measurement set $M_{d,G}$ well enough to 
identify the sub-variety corresponding to 
$M_{d-1,G}$ intrinsically.  That is the 
topic of the next section.

\section{Getting Down to $d=1$}\label{sec:flat}

In this section we will prove the following proposition:

\begin{proposition}
\label{prop:induct}
Let $G$, an ordered graph on $n \ge d+2$ vertices with $m$ edges,
be generically globally rigid in $\CC^d$
and let $H$ be some ordered graph 
on $n$ vertices
with $m$ edges.
Suppose   $M_{d,G} = M_{d,H}$.
Then 
$M_{d-1,G} = M_{d-1,H}$ and, by induction,
$M_{1,G} = M_{1,H}$. 
\end{proposition}

The basic strategy is to show that points $\x$ in $M_{d,G}\setminus M_{d-1,G}$ look intrinsically different
in $M_{d,G}$ than points $\y$ of $M_{d,G}$ that are also in $M_{d-1,G}$. This means that these cases can be 
distinguished
from the variety alone, without knowing the generating graph $G$. 
We will not simply be able to use smoothness as the distinguishing factor
as there can be points in $M_{d,G} \setminus M_{d-1,G}$ that are not
smooth.
Our characterization will
involve looking at Gauss fibers in $M_{d,G}$ in the neighborhood around such points.
Luckily, from results in~\cite{conGR,ght}, we have a reasonable understanding of these Gauss fibers
(at least generically)
and how they relate to equilibrium stresses of $(G,\p)$ and affine transformations of $\p$. 
The key distinguishing features of these points
are described in Propositions~\ref{prop:mainA} and~\ref{prop:mainB}
below. 
The geometry that distinguishes points 
in $M_{d,G}$ that are also in $M_{d-1,G}$ is illustrated schematically 
in 
Figure~\ref{fig:cone-cyl}.

In what follows, we will make the formal argument as weak as possible,
only focusing on generic points, 
but we will also add remarks as we go along, with stronger statements
for geometric intuition.

\begin{figure}[h!]
	\centering
	\subfloat[]{\adjincludegraphics[height=0.25\textheight,trim={0 {0.25\height} 0 {0.25\height}},clip]{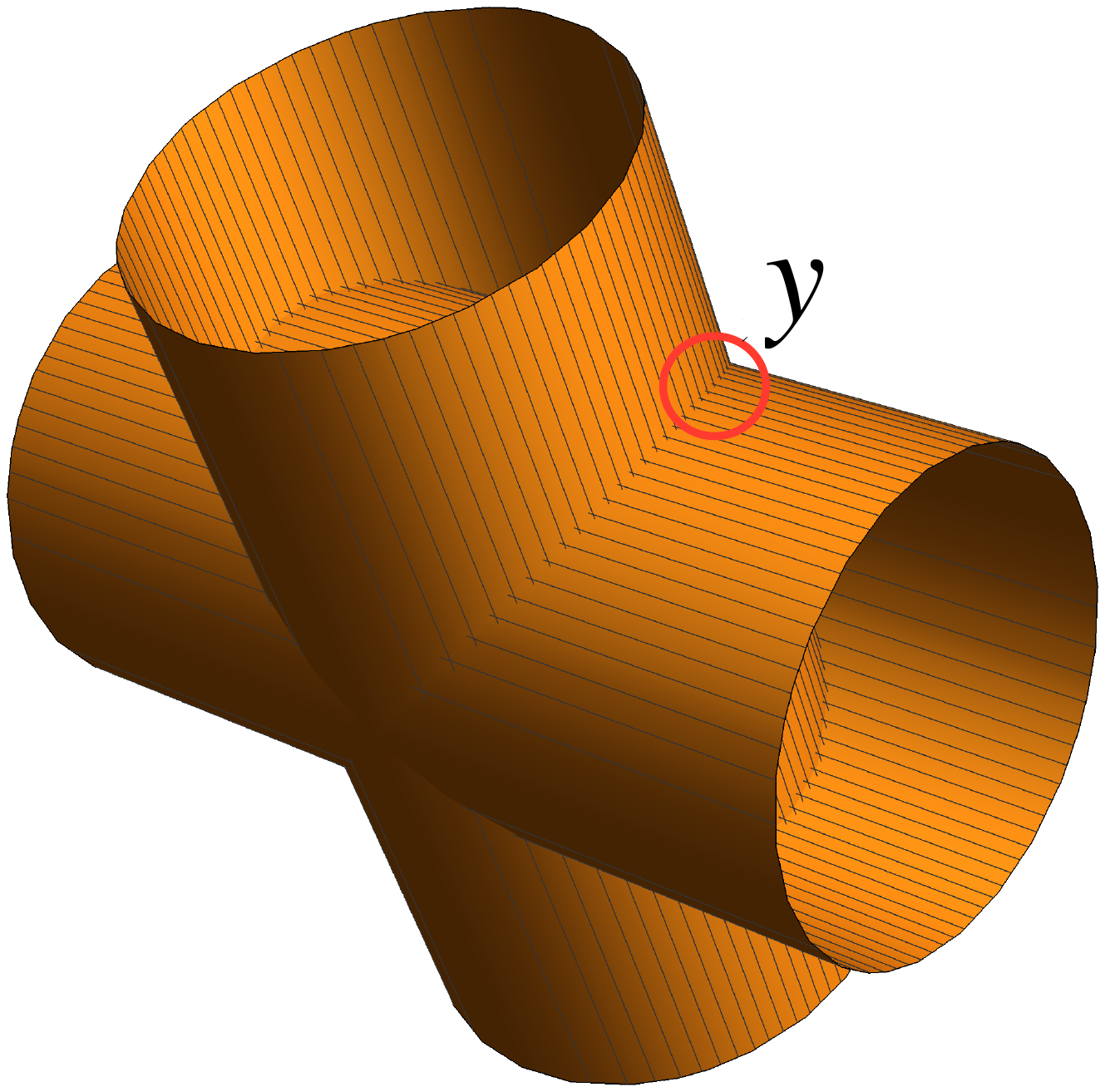}}
    \subfloat[]{\includegraphics[height=0.28\textheight]{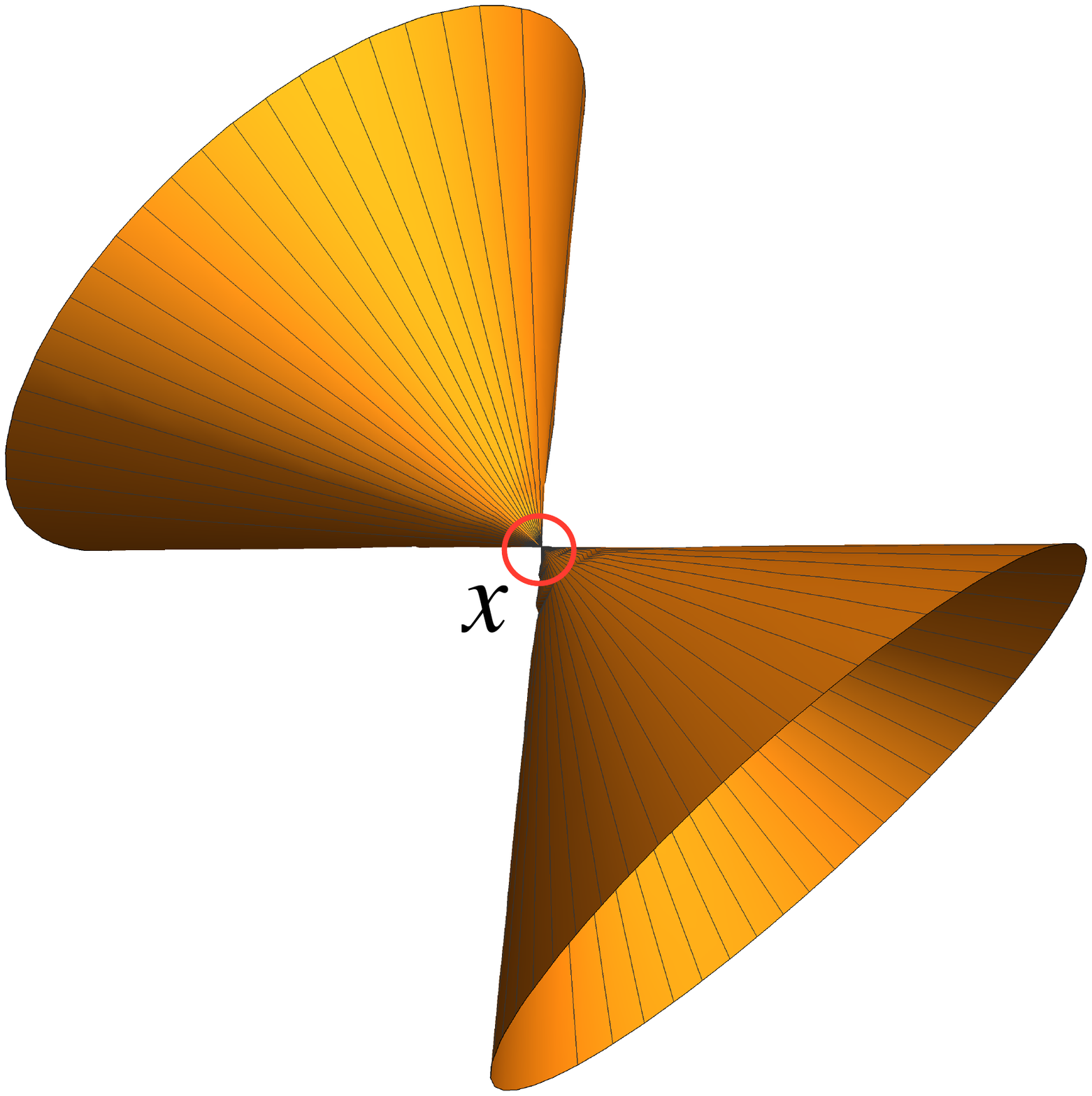}}
    
\caption{Two types of singular points on ruled varieties.  (a) 
The Gauss fibers on a variety consisting of two intersecting cylinders consist of the ruling lines indicated.  Points in the 
intersection of the two cylinders, such as the one marked $y$ 
are in the singular locus, but still lie in the closure of a finite number of 
generic Gauss fibers (in this case, one ruling line 
from each cylinder).  
Proposition \ref{prop:mainA} says that measurements $y$ that arise
from configurations with full spans are either smooth points (in a single
fiber closure)
or 
lie in the closure of a finite number of 
generic Gauss fibers as in this figure.
(b) The Gauss fibers on the elliptic cone
also consist of ruling lines, as indicated.  The cone point,
marked as $x$, lies in the closure of an infinite number of 
ruling lines.  This 
is a different situation than we saw (for $y$) in (a).
Proposition \ref{prop:mainB}
says that measurements $x$ that arise
from configurations with deficient spans 
lie in the closure of an infinite number of 
generic Gauss fibers as in this figure.
\label{fig:cone-cyl}}
\end{figure}

\begin{lemma}
\label{lem:sard}
Let $G$ be generically locally rigid in $\CC^d$,
with $n \ge d+1$ vertices.
Suppose $(G,\p)$ is an infinitesimally flexible framework. Then the point $\x:=m_G(\p)$
is not generic in $M_{d,G}$.

In particular, if $\p$ has deficient affine span, then $m_G(\p)$ is non-generic.
\end{lemma}
\begin{proof}[Proof sketch]
From Theorem~\ref{thm:glr}, $G$ is generically locally rigid
iff it is generically infinitesimally rigid iff the generic dimension
of the differential of $m_G(\cdot)$ is $dn-\binom{d+1}{2}$.

If $\x$ is not a smooth point of $M_{d,G}$ then it cannot
be generic and we are done.

Next we restrict the map $m_G(\cdot)$ 
by removing the non-smooth points from $M_{d,G}$
and then removing the preimages
of these non-smooth points from the domain.
By assumption, the configuration $\p$ is not a regular point of $m_G(\cdot)$, making $\x$ not a regular value
of its image. 

But from Sard's theorem
applied to $m_G(\cdot)$,
the 
set of critical values is  of lower dimension. 
This set is also constructible and defined over $\QQ$.
This set remains of lower dimension under
closure, 
thus the critical values must satisfy some extra equation,
making them non-generic.

By Lemma \ref{lem:not-flat}, any $(G,\p)$ with deficient affine 
span is infinitesimally flexible when $G$ has at least $d+1$ 
vertices, giving the second part of the lemma.
\end{proof}

\begin{definition}
Fix $d$ and $G$.
We say that $\x$ is an \defn{unhit} point of 
$M_{d,G}$ if there is no configuration $\p$ such that
$\x = m_G(\p)$. Otherwise it is \defn{hit}.
\end{definition}

\begin{lemma}
\label{lem:genhit}
Let $G$ be an ordered graph on $n\ge d+1$ vertices that is
generically locally rigid in $\CC^d$.
Let $\x$ be generic in $M_{d,G}$. Then 
$\x$ is hit. Moreover, any configuration $\p$ hitting $\x$
is infinitesimally rigid and has full affine span.
\end{lemma}
\begin{proof}
The hit set is an irreducible 
constructible set with  $M_{d,G}$ as its closure.
By Lemma  \ref{lem:milne}, it must then contain a non-empty
(Zariski) 
open subset of 
$M_{d,G}$. 
Thus the unhit set must be contained
in a 
closed subset (i.e., a subvariety).  
This renders all unhit points non-generic.

Infinitesimal rigidity follows from Lemma~\ref{lem:sard}, which also gives us the
stated span.
\end{proof}

\begin{definition}
Let $\p$ be a configuration in $\CC^d$ with a full affine span. Then 
the \defn{open affine class} $\CT(\p)$ is the set of configurations that are  
affine
images of $\p$, and are non-degenerate (have full span).  
An affine class is \defn{generic} if it contains
a generic configuration.  (Generic affine classes exist, since $\CT(\p)$ is defined for every $\p$ with full span.)

Given a generic affine class $\CT$, we define the generic locus $\CT^g$ to be the subset of configurations in $\CT$ that are also generic as configurations.

Let $\overline{\CT(\p)}$ be the closure of an affine class.
This includes the degenerate affine images.
$\overline{\CT(\p)}$ is a linear space. 
\end{definition}

\begin{lemma}
\label{lem:linim}
Let $G$ be  an ordered graph on $n$ vertices and $\p$ a configuration 
of $n$ points in $\CC^d$.
Then $m_G(\overline{\CT(\p)})$ is 
a linear space,  and in particular, it is 
closed.
\end{lemma}
\begin{proof}
For each edge $ij$ of $G$, define its edge vector as
$\e:= \p_i-\p_j$ in $\CC^d$.
Then the complex squared length on that edge is the vector product
$\e^t \e$.

An affine transform, $A$,  applied to  $\p$ 
can be expressed as 
$\p_i \rightarrow \M \p_i + \t$, where $\M$ is some $d$-by-$d$ complex matrix and $\t$
is some (translation) vector. The effect on each edge vector is of
the form $\e_{ij} \rightarrow \M \e_{ij}$.
The effect on its squared length is 
$\e^t \e \rightarrow
\e^t \M^t \M \e =: \e^t \Q \e = \tr(\Q \e\e^t)$, 
where $\Q$ is a symmetric matrix. 
Note that the rightmost
expression is linear in $\Q$.

Since we are in the complex setting,
using a Takagi 
factorization, 
every symmetric matrix $\Q$ arises in this form from some $\M$.

Thus, we can model the action of $m_G(\cdot)$ on $\overline{\CT(\p)}$
by defining a map 
$n_{G,\p}(\Q)$ from symmetric $d\times d$ matrices $\Q$ to $\CC^m$ that  
acts 
coordinate-wise as 
$n_{G,\p}(\Q)_{ij} := \tr(\Q \e_{ij}\e_{ij}^t)$.  Since $n_{G,\p}(\cdot)$ 
is a linear map acting on the linear space of symmetric matrices, 
its image, which is $n_G(\overline{\CT(\p)})$, is a linear subspace of $\CC^m$ as claimed.
\end{proof}

\begin{definition}
Let $V$ be an irreducible homogeneous variety. 
We define an (open) \defn{Gauss fiber} $F$ of $V$ to be a maximal set
of smooth points of $V$ with a common tangent space.
(For an inhomogeneous variety, we would instead have to work with
affine tangent planes.)
We say that $F$ is a 
\defn{generic Gauss fiber} if it contains a point
that is generic in $V$. 
Given a generic Gauss fiber $F$ of $V$, we define the 
generic locus $F^g$ to be the subset of points
in $F$ that are also generic in $V$.
\end{definition}
The term ``Gauss fiber'' is used as it is the
fiber above a point in the image of the (rational) Gauss map
$\x\mapsto T_\x V$,
taking each smooth point of $V$ to the appropriate Grassmanian.
Importantly, the definition of $F$ and $F^g$ 
only depend on the geometry of the variety $V$, and not on any other information (such as how $V$ may have been generated from some graph).

\begin{remark}
\label{rem:rv1}
A deeper result about ruled varieties states that
if $F$ is a generic Gauss fiber of any irreducible homogeneous variety, 
then its closure,
$\overline{F}$,
is always 
a
linear space~\cite[Section 2.3.2]{rv}, and in particular,
irreducible.
This also tells  us that $F^g$ is 
dense in $F$ (Lemma~\ref{lem:dense}) and so 
$\overline{F^g}=\overline{F}$.
\end{remark}

The next set of lemmas will establish a correspondence 
between generic Gauss fibers of $M_{d,G}$ and
affine classes of configurations.

\begin{definition}
Let $V$ be a homogeneous variety in $\CC^m$.
Let $\x$ be a smooth point in $V$.
Let $\phi$ be a non-zero element of $(\CC^m)^*$.
We say that $\phi$ is \defn{tangent} to $V$ at $\x$ if 
$T_{\x}V \subseteq \ker(\phi)$. We will call (with slight abuse
of duality)
such a 
$\phi$ a \defn{tangential hyperplane}.
\end{definition}

The following lemma relates an equilibrium stress vector
for $(G,\p)$ to the geometry of $M_{d,G}$ around $m_G(\p)$.
\begin{lemma}[{\cite[Lemma 2.21]{ght}}]
\label{lem:stTan}
Let $G$ be an ordered graph with $n \ge d+2$ vertices.
Let $(G,\p)$ be an infinitesimally rigid  framework  with 
$m_G(\p)$ smooth in 
$M_{d,G}$ (such as when $\p$ is generic).
A non-zero $\omega \in (\CC^m)^*$ 
is tangent
to $M_{d,G}$ at $m_G(\p)$
iff  $\omega$ is an equilibrium stress for $(G,\p)$.
\end{lemma}

\begin{remark}
\label{rem:extra}
If $m_G(\p)$ is smooth, but $(G,\p)$ is infinitesimally 
flexible, then every tangential hyperplane $\omega$
is still an equilibrium stress for $(G,\p)$, but the
framework will also satisfy extra equilibrium stresses.
\end{remark}

\begin{lemma}
\label{lem:ftoa}
Let $G$, an ordered graph on $n \ge d+2$ vertices with $m$ edges,
be generically globally rigid in $\CC^d$.
Let $F$ be a generic Gauss fiber of $M_{d,G}$.
Then
there exists a single affine class $\CT$ such that 
all $\p$ with $m_G(\p) \in F^g$ are in $\CT$; i.e., 
$m^{-1}_G(F^g)\subseteq \CT$. 
This class $\CT$ is generic. 
\end{lemma}
\begin{proof}
From Lemma~\ref{lem:genhit}, 
each $\x \in F^g$ is hit, 
giving us at least one
$\p$ with $m_G(\p)=\x$. 
Also from Lemma~\ref{lem:genhit}, 
each such $\p$ is infinitesimally rigid and thus
has a full span.

Lemma~\ref{lem:stTan} then tells us that the equilibrium 
stresses for $(G,\p)$  with $m_G(\p) \in F^g$
correspond
to the tangential hyperplanes at $m_G(\p)$.
Since the tangents, and thus 
tangential hyperplanes, agree for all $\x \in F^g$,
all such $\p$ share the same space $S$ of equilibrium stresses.

From Lemma~\ref{lem:preG}, above  any 
$\x \in F^g$ there is a generic configuration $\q$ 
and from Theorem~\ref{thm:sharedKernel} 
(see also Remark~\ref{rem:shared})
$\q$ has a shared stress kernel of dimension $d+1$.
Thus $S$ must have a shared stress kernel of dimension $d+1$.
This makes the
dimension of the set of $d$-dimensional configurations 
having this stress space $S$ 
equal to $d(d+1)$.
In particular, 
this places all such $\p$ in some unique closed affine class $\bar{\CT}$.
This, along with the established affine span of
$\p$ places it in $\CT$.
Genericity of $\CT$ comes from the genericity of $\q$.
\end{proof}

In light of Lemma~\ref{lem:ftoa}, the following  is well-defined.
\begin{definition}
\label{def:AofF}
Let $G$, an ordered graph on $n \ge d+2$ vertices with $m$ edges,
be generically globally rigid in $\CC^d$.
Let $F$ be a generic Gauss fiber of $M_{d,G}$. Define, 
by an overloading of notation, 
$\CT(F)$ to be the generic affine
class $\CT(\p)$ for any/every $\p$ above any 
$\x \in F^g$.  We also denote by $\CT(\cdot)$
the map $F\mapsto \CT(F)$, which is defined
for generic Gauss fibers of $M_{d,G}$.
\end{definition}

\begin{remark}
\label{rem:anysmooth}
From Remark~\ref{rem:extra}, when $G$ is generically
globally rigid and 
$\q$ is any configuration so that $m_G(\q)$
is smooth and in a generic Gauss fiber $F$ (even if $m_G(\q)$ is not in $F^g$), 
then $\q \in \overline{\CT(F)}$. 
Additionally, any such $\q$ must have an equilibrium 
stress matrix of rank $n-d-1$.
If additionally, 
$\q$ has
a full affine span, then $\q \in \CT(F)$. (Later we will see that such $\q$,
with $m_G(\q)$ smooth, must
in fact
always have full affine span.)
\end{remark}

\begin{lemma}
\label{lem:ftoaInj}
Let $G$, an ordered graph on $n \ge d+2$ vertices with $m$ edges,
be generically globally rigid in $\CC^d$.
If $\p^1$ and $\p^2$ are generic configurations
and $\CT(\p^2) =\CT(\p^1)$, then $m_G(\p^1)$ and $m_G(\p^2)$
are both generic and 
in the same generic Gauss fiber of $M_{d,G}$.
Thus, if $F^1$ and $F^2$ are  two 
different generic Gauss fibers,  then $\CT(F^1) \neq \CT(F^2)$.
\end{lemma}
\begin{proof}
We proceed as in the proof of Lemma~\ref{lem:ftoa}.
Since $\p^1$ and $\p^2$ 
are  non-singular affine images of each other,
they must satisfy all of the same equilibrium 
stress matrices. Thus $m_G(\p^1)$ and $m_G(\p^2)$ 
must have the same tangential hyperplanes, and be in the same
Gauss fiber $F$ of $M_{d,G}$.
From Lemma~\ref{lem:genMap}, the images $m_G(\p^i)$, $i=1,2$ 
are generic, so $F$ is  a generic Gauss fiber.
\end{proof}
We get the following corollary, which  is also
interesting in its own right.
\begin{proposition}\label{prop:GFs-bij-to-ACs}
Let $G$, an ordered graph on $n \ge d+2$ vertices with $m$ edges,
be generically globally rigid in $\CC^d$.  The
map $\CT(\cdot)$ gives 
a 
bijection between 
generic Gauss fibers of $M_{d,G}$
and
generic affine classes.
Finally, we have
$F^g=m_G(\CT(F)^g)$.
\end{proposition}
\begin{proof}
Lemma \ref{lem:ftoaInj} implies that the map $\CT(\cdot)$
from generic Gauss fibers of $M_{d,G}$ to affine classes is injective.
Lemma \ref{lem:preG} also implies that if $F$ is a generic 
Gauss fiber of $M_{d,G}$ that $\CT(F)$ is a generic affine class.

The map $\CT(\cdot)$ is also surjective.  
By definition, a generic affine class arises 
as $\CT(\p)$ for a generic configuration $\p$. By Lemma 
\ref{lem:genMap}, the image $m_G(\p)$ is generic 
in $M_{d,G}$.  Hence the Gauss fiber containing $m_G(\p)$ 
is generic.  Since $\CT(\p)$ was an arbitrary generic affine 
class, we have surjectivity.

From Lemma~\ref{lem:ftoa} we have 
$m^{-1}_G(F^g) \subseteq \CT(F)$. Since, from
Lemma~\ref{lem:genhit}, each point in $F^g$ is hit, this gives
us 
$F^g \subseteq m_G(\CT(F))$. 
From Lemma~\ref{lem:preG}, this means
$F^g \subseteq m_G(\CT(F)^g)$. 

In the other direction, Lemma~\ref{lem:ftoaInj} gives us 
$F \supseteq m_G(\CT(F)^g)$. 
From Lemma~\ref{lem:genMap}, this means
$F^g \supseteq m_G(\CT(F)^g)$. 
\end{proof}

The following is the main structural 
lemma that we will need going forward.

\begin{lemma}
\label{lem:closure}
Let $G$, an ordered graph on $n \ge d+2$ vertices with 
$m$ edges, be generically globally rigid in $\CC^d$.
Let $F$ be a generic Gauss fiber of $M_{d,G}$. 
Then  $\overline{F^g} = m_G(\overline{\CT(F)})$.
\end{lemma}
\begin{proof}
From Proposition~\ref{prop:GFs-bij-to-ACs}
we have
$F^g = m_G(\CT(F)^g)
\subseteq  m_G(\CT(F))$. 

From Lemma~\ref{lem:dense}, 
${\CT(F)^g}$ is dense in 
$\overline{\CT(F)}$ and so 
$\overline{\CT(F)^g}=\overline{\CT(F)}$.
From continuity, we have 
$\overline{m_G(\CT(F)^g)}
\supseteq m_G(\overline{\CT(F)^g})$. 
Thus 
$\overline{F^g} = \overline{m_G(\CT(F)^g)}
\supseteq m_G(\overline{\CT(F)^g})
=
m_G(\overline{\CT(F)})
$. 

For the other direction, 
we have established above that 
$F^g 
\subseteq  m_G(\CT(F))$. 
Meanwhile, 
from Lemma~\ref{lem:linim}, the image 
$m_G(\overline{\CT(F)})$
is 
closed, and thus, from continuity, 
$
\overline{m_G(\CT(F))} =
m_G(\overline{\CT(F)})$.
Thus 
$\overline{F^g} \subseteq 
\overline{m_G(\CT(F))} =
m_G(\overline{\CT(F)})$.
\end{proof}
\begin{remark}
\label{rem:rv2}
In light of Remark~\ref{rem:rv1}, 
we see that for a generically globally rigid graph $G$ and generic Gauss fiber $F$ of $M_{d,G}$, we actually have
 $\overline{F} =m_G(\overline{\CT(F)})$. This also means that 
all points of $\overline{F}$ are hit.
\end{remark}

\begin{lemma}
\label{lem:conv}
Let $G$, an ordered graph on $n \ge d+2$ vertices with $m$ edges,
be generically globally rigid in $\CC^d$.
Let $\x$ be any point 
(not necessarily generic)
in
 $M_{d,G}\setminus M_{d-1,G}$ and in $\overline{F^g}$, 
for some generic Gauss fiber $F$ of $M_{d,g}$.
There must be a configuration $\q$ that has full span and is in ${\CT(F)}$
such that $m_G(\q)=\x$.
For such a $\q$, the framework $(G,\q)$ 
must be infinitesimally rigid, and 
hence also be locally rigid.
\end{lemma}
\begin{proof}
From Lemma~\ref{lem:closure} there must be some $\q$ in 
$\overline{\CT(F)}$ such that 
$m_G(\q)=\x$.
From the assumption that  $\x$ is not in 
$M_{d-1,G}$, we know that 
$\q$ must have an affine span of
dimension $d$. Thus $\CT(\q)$ is well defined and
is equal to $\CT(F)$.

If $(G,\q)$ were  infinitesimally flexible, then from
Lemma~\ref{lem:ir-aff} so too would all of the 
points in 
$\CT(\q)$, which equals $\CT(F)$.  
But from Lemma~\ref{lem:sard}, this would  
contradict the assumed genericity of $F$.
Local rigidity  follows 
from infinitesimal rigidity and  Theorem \ref{thm:ir-lr}.
\end{proof}

The next two propositions form the central part of our argument.
Informally, they say that we can distinguish between 
points in the measurement set that arise from lower-dimensional
configurations from those that are merely singular by looking 
at the generic Gauss fibers going through them.  Figure \ref{fig:cone-cyl} gives a schematic of the two situations.

\begin{proposition}
\label{prop:mainA}
Let $G$ be an ordered  graph with $n \ge d+2$ vertices
 that is generically globally rigid in $\CC^d$.
Let $\x$ be any point (not necessarily generic)
in
 $M_{d,G}\setminus M_{d-1,G}$. Then there are at most a finite number
of generic Gauss fibers $F$ of $M_{d,G}$ with $\x$ 
in $\overline{F^g}$.
\end{proposition}
Note that if $\x$ is a smooth point of $M_{d,G}$ then it is in a single generic 
Gauss fiber closure. 
But here, we are not making such assumptions on $\x$;
for example, we will allow for $\x$ that are measurements of frameworks that are (due to non-genericity) not globally rigid. This can occur
even in a generic affine class~\cite[Example 8.3]{cone}.

Informally, the key idea is that if $\x$ is in an infinite number of 
$\overline{F^g}$, then it has preimage configurations 
from an infinite number
of affine classes. 
From the full span  assumption, this gives us an infinite number
of preimage configurations, unrelated by congruence.  Each 
of these preimage configurations will have to be locally 
rigid from the assumed genericity of each Gauss fiber.
This would contradict the fact that, 
up to congruence, there can only be a finite 
number of locally rigid 
configurations with the same edge lengths.
\begin{proof}[Proof of Proposition \ref{prop:mainA}]
Let $\{F^i | i \in I\}$ be the collection of generic Gauss fibers
of $M_{d,G}$ containing $\x$ in their closures.  
A priori, the index set $I$ might be infinite.
For every such $F^i$, 
from Lemma~\ref{lem:conv}, 
there must be a configuration $\q^i$ in $\CT(F^i)$
that has full span, is 
locally rigid
and such that $m_G(\q^i)=\x$.
Since $\q$ has full affine span,  $\CT(\q^i)$ is well defined and is 
equal to $\CT(F^i)$.
From Lemma~\ref{lem:ftoaInj}, for any two such distinct
$F^i$ and $F^j$, we have $\CT(\q^i) \neq \CT(\q^j)$.
Thus $\q^i$ cannot be congruent to $\q^j$.

Suppose there were an infinite number of $F^i$.
Then there would be an infinite number of 
locally rigid congruence classes $[\q^i]$ that map to $\x$.
But this contradicts Lemma~\ref{lem:ccomp}.
\end{proof}

\begin{proposition}
\label{prop:mainB}
Let $G$, an ordered graph on $n \ge d+2$ vertices with $m$ edges,
be generically globally rigid in $\CC^d$.
There is an $\x$, generic in $M_{d-1,G}$, such 
there are an infinite number of Gauss fibers $F$ of $M_{d,G}$
with $\x$ in  $\overline{F^g}$.
\end{proposition}
The key idea is that if $\q$ is a configuration with a deficient affine span,
then there are an infinite number of configurations $\p$ with  full affine
spans such that $\q \in \overline{\CT(\p)}$. This will give us an infinite
number of generic Gauss fibers with $\x=m_G(\q)$ in their closures.
\begin{proof}
We start with a generic configuration $\p$.
Let $\pi$ be the projection
from $d$-dimensional configurations to $d-1$ dimensional configurations
that simply ignores 
the last spatial coordinate.
Let $\q := \pi(\p)$ and $\x := m_G(\q)$.
Since $\p$ is generic as a $d$-dimensional configuration, 
$\q$ is generic as a $(d-1)$-dimensional configuration, and
$\x$ is generic in $M_{d-1,G}$.

Let $F$ be the Gauss fiber of $M_{d,G}$ that contains $m_G(\p)$.
Since $\q \in \overline{\CT(F)}$, 
from Lemma~\ref{lem:closure}, $\x$ is in 
$\overline{F^g}$.
This gives us one fiber $F$ for the 
proposition. Now we show how to get more.

Define $\CL(\p):= \pi^{-1}(\q)$ to be the space of lifts of $\q$.
The space of lifts is an affine space that contains $\p$, and so,
by Lemma \ref{lem:dense}, $\CL(\p)$
contains a dense set of generic configurations.
Since $n \ge d+2$, we can find
an infinite number of $\p'$
that are generic configurations, are in 
$\CL(\p)$ and with each in a different
affine class.
(Any finite number of affine classes are contained in 
a finite number of strict subvarieties of the linear 
lifting space space, and thus cannot cover all of the
generic configurations).

For any such configuration, say $\p'$, that is not in $\CT(\p)$,
we can apply the same argument to get 
another Gauss fiber $F'$
with $\x$ in $\overline{F'^g}$.
From Lemma~\ref{lem:ftoa}, we have $F \neq F'$.
Since we can do this endlessly, we obtain our infinite number of 
fibers for $\x$.
\end{proof}

\begin{remark}
 Any $\x$ meeting the hypotheses of Proposition~\ref{prop:mainB}
 cannot be a smooth point in $M_{d,G}$ (as it is in the
 closure of multiple Gauss fibers, and the Gauss map,
 where defined,
 is continuous).
 Since this $\x$ is
 also generic in $M_{d-1,G}$, we can conclude that
 all of $M_{d-1,G}$ lies in the singular locus of 
 $M_{d,G}$. 
\end{remark} 
\begin{remark}
To recap, 
every $\q$ with $m_G(\q)$ smooth and in
a generic Gauss fiber $F$ of a generically globally rigid 
graph $G$, has a full affine span,
is infinitesimally rigid, 
and 
is in $\CT(F)$. Such a framework $(G,\q)$ must
be globally rigid~\cite{conGR}.

The closure of $F$ is the linear space
$m_G(\overline{\CT(F)})$. 
This means that 
all points in
$\overline{F}$ are hit.
It also means that, for each point $\x \in \overline{F}$, there
must be some point $\q$ in $\overline{\CT(F)}$
with $m_G(\q)=\x$.

Points that are  smooth in $M_{d,G}$
are in only one Gauss fiber and one
Gauss fiber closure.

Let $F$ be a generic Gauss fiber.
Let the ``bad'' points be  $B:= \overline{F}\setminus F$.
These are non-smooth in $M_{d,G}$.
Points that are in $B$ due to deficient
span will, generically, 
be in the closure of infinitely many distinct generic Gauss fibers. 
All ``other'' points in $B$
(no deficient span in the pre-image)
can be in only a 
finite number of generic $\overline{F}$.
(Any of the preimages $\q$ of these other bad points 
is also infinitesimally rigid.)
These other bad points can occur, say, when 
$(G,\q)$ is not globally rigid. 
(Note that global rigidity
of frameworks is not an 
affine invariant property~\cite[Example 8.3]{cone}.)

At non-generic Gauss fibers $F$ of $M_{d,G}$, most bets are off.
$F$ can be of some larger dimension, and conceivably 
be reducible. It is even conceivable that there are 
$\q$ that have full spans
 and are infinitesimally flexible but such that 
 $m_G(\q)$ 
 is still a smooth point of $M_{d,G}$
 (in such a non-generic $F$).
\end{remark}

The next lemmas will let us treat $G$ and $H$ symmetrically.  We 
start with a technical result about Gauss fiber dimension that 
allows us to identify generic global rigidity intrinsically
in $M_{d,G}$.
\begin{lemma}\label{lem:gauss-fiber-dim}
Let $G$ be a generically locally rigid 
graph with $n\ge d + 2$ vertices and $m$ edges.  Suppose 
that, at some generic $\p$ the shared stress kernel of $(G,\p)$ has 
dimension $k\ge d+1$.  Let $F$ be the Gauss fiber of $M_{d,G}$ 
containing $m_G(\p)$.  Then the dimension of $\overline{F^g}$
is $dk - \binom{d+1}{2}$.

\end{lemma}
\begin{proof}
Let $\p$ be a generic configuration.  Define $K$ 
to be the space of configurations $\q$ such that 
$(G,\q)$ satisfies all the equilibrium stresses
$(G,\p)$ does.  This $K$ is a linear space of 
dimension $dk$.  

As discussed in 
Definition \ref{def:stress-kernel}, $k$, the dimension of the 
shared stress kernel, is at least $d+1$.
Hence, $K$ is of dimension at least $d(d+1)$. Moreover
(see Definition~\ref{def:stress})
$K$ must
include all affine images of $\p$, including
all $\q$ that are congruent to $\p$.

Let $K^g$ be the configurations
in $K$ that are generic in configuration space, From Lemma~\ref{lem:dense}, 
$K^g$ is a dense subset of $K$, which is closed, and so 
$\overline{K^g} =K$.  

Let 
$F$ be the Gauss fiber of $M_{d,G}$ containing $m_G(\p)$; $F$
is generic because $\p$ is.  Now that we know $F$ is a generic
affine class, we can show that  $F^g = m_G(K^g)$
by following the proof of the same statement in 
Proposition~\ref{prop:GFs-bij-to-ACs}.
This gives us 
$\overline{F^g}=
\overline{m_G(K^g)}=
\overline{m_G(\overline{K^g})}=
\overline{m_G(K)}$.
The second equality
is due to continuity.

From the above, we get
$\dim(\overline{F^g})=
\dim(\overline{m_G(K)}) = 
\dim(m_G(K))$.  The second equality
is due the fact that a constructible set
and its Zariski closure have the same 
dimension.
(In the case
that $G$ is generically globally rigid,
then $K=\overline{\CT(F)}$ and due to
Lemmas~\ref{lem:linim} and \ref{lem:closure} we actually have
$\overline{F^g} = m_G(K)$.)

By generic local rigidity of $G$, we 
know that the maximum rank of 
$dm_G$, restricted to its action on $K$, is 
$dk - \binom{d+1}{2}$.  
(Recall that $K$ includes all configurations that
are congruent to $\p$.
It follows that, 
at  $\p$, the 
differential of $m_G$, acting
as a map restricted to $K$, 
has a kernel of
dimension at least $\binom{d+1}{2}$, corresponding to the trivial infinitesimal flexes.
If this kernel at our generic $\p$
were any larger, $(G,\p)$
could not be infinitesimally rigid.)

Using the constant rank theorem and Sard's Theorem as 
in the proof of Lemma \ref{lem:infrig}, we conclude that 
the dimension of $\overline{F^g}$ is 
exactly this size.
\end{proof}
\begin{remark}
Using the fact that, in fact, $F$ is irreducible
(see Remark \ref{rem:rv1}), 
we can improve Lemma \ref{lem:gauss-fiber-dim} to 
give the dimension of $\overline{F}$ 
instead of 
$\overline{F^g}$.
\end{remark}

We now interpret Lemma~\ref{lem:gauss-fiber-dim}
in terms of generic global rigidity.
\begin{lemma}
\label{lem:both}
Let $G$, an ordered graph on $n \ge d+2$ vertices with $m$ edges,
be generically globally rigid in $\CC^d$.
Let $H$ be some 
ordered graph also with $n$ vertices
with $m$ edges.
Suppose   $M_{d,G} = M_{d,H}$.
Then $H$ is also generically globally rigid in $\CC^d$.
\end{lemma}
\begin{proof}
Since $M_{d,G} = M_{d,H}$, they, in particular have the same 
dimension.  If $G$ is generically globally rigid, then it is also
generically locally rigid.  By Lemma \ref{lem:infrig} 
this implies that $H$ 
(having the same number of vertices as $G$)
is also generically locally rigid.  

Because $M_{d,G}$ and $M_{d,H}$ are the same variety, 
the generic Gauss fiber $F$ of $M_{d,G}$ containing 
$\x := m_G(\p)$ is also the generic Gauss fiber of 
$M_{x,H}$ containing $\x$
(and $\x$ will have some  other generic preimage $\q$ under $m_H(\cdot)$ by Lemma \ref{lem:preG}).

Lemma \ref{lem:gauss-fiber-dim} then lets us compute 
the dimension of $\overline{F^g}$ two different ways,
using the shared stress kernel of $(G,\p)$ and $(H,\p)$,
respectively.  Since $G$ is generically globally 
rigid, Theorem \ref{thm:sharedKernel} implies its
shared stress kernel is $(d+1)$-dimensional.  To 
avoid a contradiction from Lemma \ref{lem:gauss-fiber-dim},
$(H,\p)$ must also have a shared stress kernel of 
dimension $d+1$, in which case Theorem \ref{thm:sharedKernel}
implies that $H$ is also generically globally rigid.%
\end{proof}


With all this in hand, we can complete the proof
the main proposition
of this section.

\begin{proof}[Proof of Proposition~\ref{prop:induct}]
From Proposition~\ref{prop:mainB}, there is a point $\x$ generic in 
$M_{d-1,G}$ in an infinite number of $\overline{F^g}$.
From Lemma~\ref{lem:both} and Proposition
\ref{prop:mainA}, $\x$ must also be in $M_{d-1,H}$. 
Since $\x$ is generic, this implies that 
 $M_{d-1,G} \subseteq M_{d-1,H}$ (otherwise the equations of $M_{d-1,H}$ would
certify $\x$ as non-generic).
From Lemma~\ref{lem:both} we can apply the same argument in the other
direction to conclude
 $M_{d-1,H} \subseteq M_{d-1,G}$.

For the induction, we use Lemma~\ref{lem:descend}.
\end{proof}

\section{$d=1$}\label{sec:d=1}

In this section, we prove the following proposition:
\begin{proposition}
\label{prop:iso}
Let $G$ and $H$ be 
ordered graphs, with $\ge 3$ vertices.
Suppose that
$G$ is $3$-connected,
$H$ has no isolated vertices, 
and 
 $M_{1,G} =M_{1,H}$.
Then
there is a vertex relabeling under which $G=H$.
\end{proposition}
This establishes the base case for an inductive 
proof of  Theorem \ref{thm:iso},
which we prove at the end of the section.

\begin{definition}\label{def:dependent-axes}
Let $V\subseteq \CC^N$ be an irreducible affine variety.
Let $L$ be a linear subspace.  Let
$\pi_{L}$ denote the quotient map taking 
$\CC^N$ to $\CC^N/L$.

Let $[N] = \{1, \ldots, n\}$.
For each $I\subseteq [N]$ the 
\defn{coordinate subspace} $S_I$ is 
the linear span of the coordinate vectors indexed by $I$; 
i.e., $S_I = \lin\{e_i : i\in I\}$.  Define 
$\overline{I} := [N]\setminus I$ for $I\subseteq [N]$.
A coordinate subspace $S_I$ is \defn{independent}
in $V$ if the dimension of $\pi_{S_{\overline{I}}}(V)$ 
is $|I|$. Otherwise $S_I$ is \defn{dependent} in $V$.
\end{definition}

\begin{lemma}
\label{lem:ind}
Let $E'$ be subset of the edges of $G$. 
$S_{E'}$ is independent in $M_{1,G}$ iff the edges of $E'$ form a 
forest over the vertices of $G$.
\end{lemma}
\begin{proof}
If $E'$ is a forest, then  given any target measurement values
in  $\CC^{|E'|}$, we
can traverse the forest and sequentially place the vertices 
in $\CC^1$ to achieve this measurement.

Conversely, if $E'$ is not a forest, then it contains a cycle $C$.
The sum of the vectors connecting the points of $C$ 
in $\CC^1$
must sum to zero,
giving us a non-trivial equation that must be satisfied.
\end{proof}

\begin{definition}
A subset of edges $E'$ of a graph $G$ is \defn{cycle supported}
if the edges of $E'$, in some order, form a simple cycle in $G$.
An edge bijection $\sigma$ between two graphs $G$ and $H$, 
is a  \defn{cycle isomorphism} if
it maps cycle supported sets, and only cycle supported sets, to  
cycle supported sets.
\end{definition}

\begin{lemma}
\label{lem:cycle}
Let $G$ and $H$ be  ordered graphs with $m$ edges.
Suppose $M_{1,G} =M_{1,H}$.
Then the mapping 
taking the ordered edges of $G$ to the ordered edges of $H$
is a cycle isomorphism.
\end{lemma}
\begin{proof}
Suppose the mapping  is not a cycle isomorphism. Then 
wlog, there is
a set of edges  $C$ that form
a simple cycle in $G$ and not $H$.

Suppose that this $C$ forms a forest in $H$.
Then from Lemma~\ref{lem:ind}, we have
$\pi_{S_{\overline{C}}}(M_{1,G})
\neq \pi_{S_{\overline{C}}}(M_{1,H})$ and thus
$M_{1,G} \neq M_{1,H}$.

Suppose instead that this $C$ is neither a simple cycle nor a forest in $H$,
then there must be an edge $e$ such that $C':=C-e$ is not a forest
in $H$, while $C'$
is a forest in $G$ (a simple cycle minus one edge is a path).
Then from Lemma~\ref{lem:ind}, we have
$\pi_{S_{\overline{C'}}}(M_{1,G})
\neq \pi_{S_{\overline{C'}}}(M_{1,H})$ and thus
$M_{1,G} \neq M_{1,H}$.
\end{proof}

\begin{figure}
\begin{center} \begin{tabular}{cccc}
  {  \includegraphics[scale=.16]{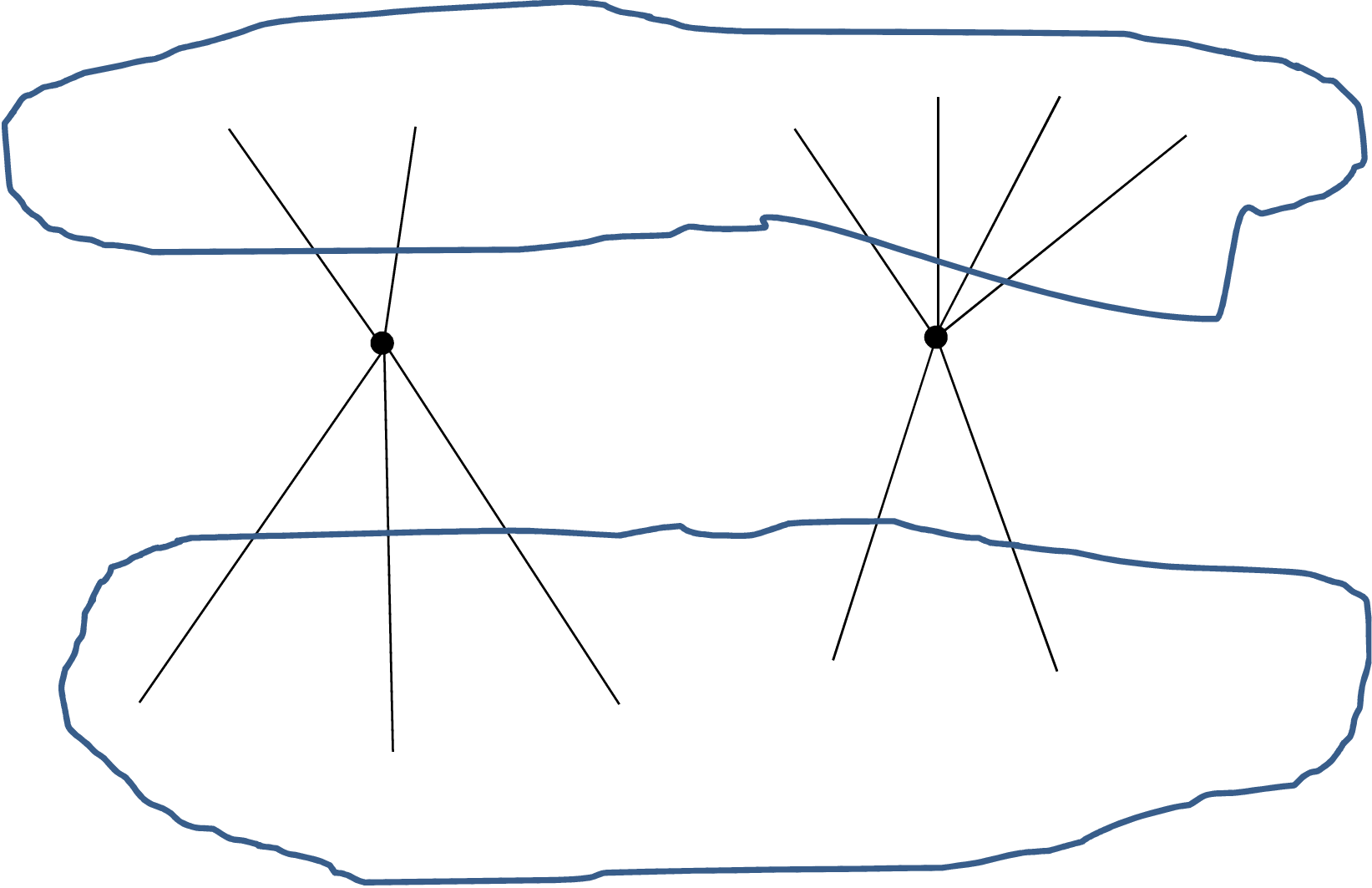}}&
  {  \includegraphics[scale=.16]{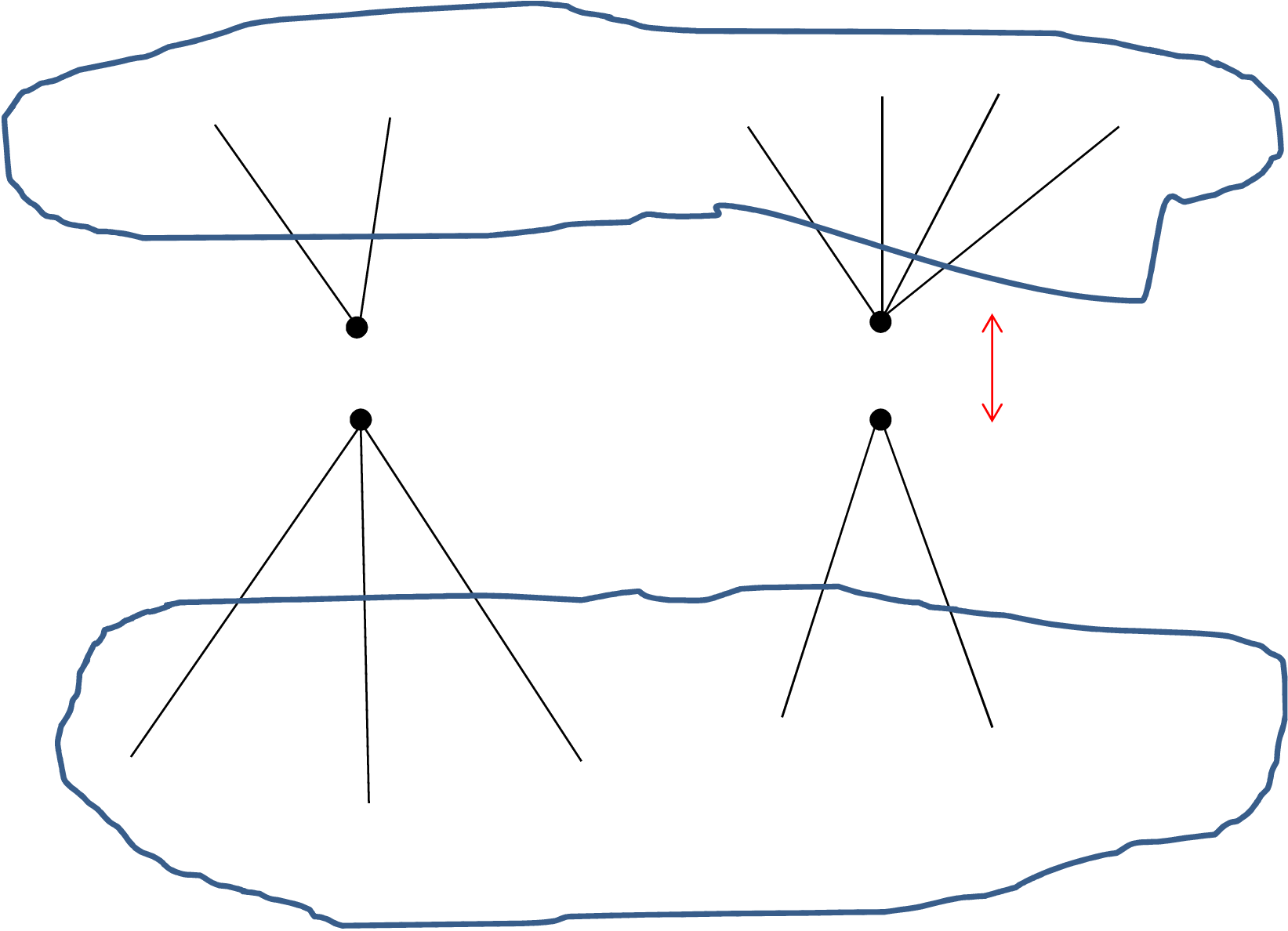}}&
  {  \includegraphics[scale=.16]{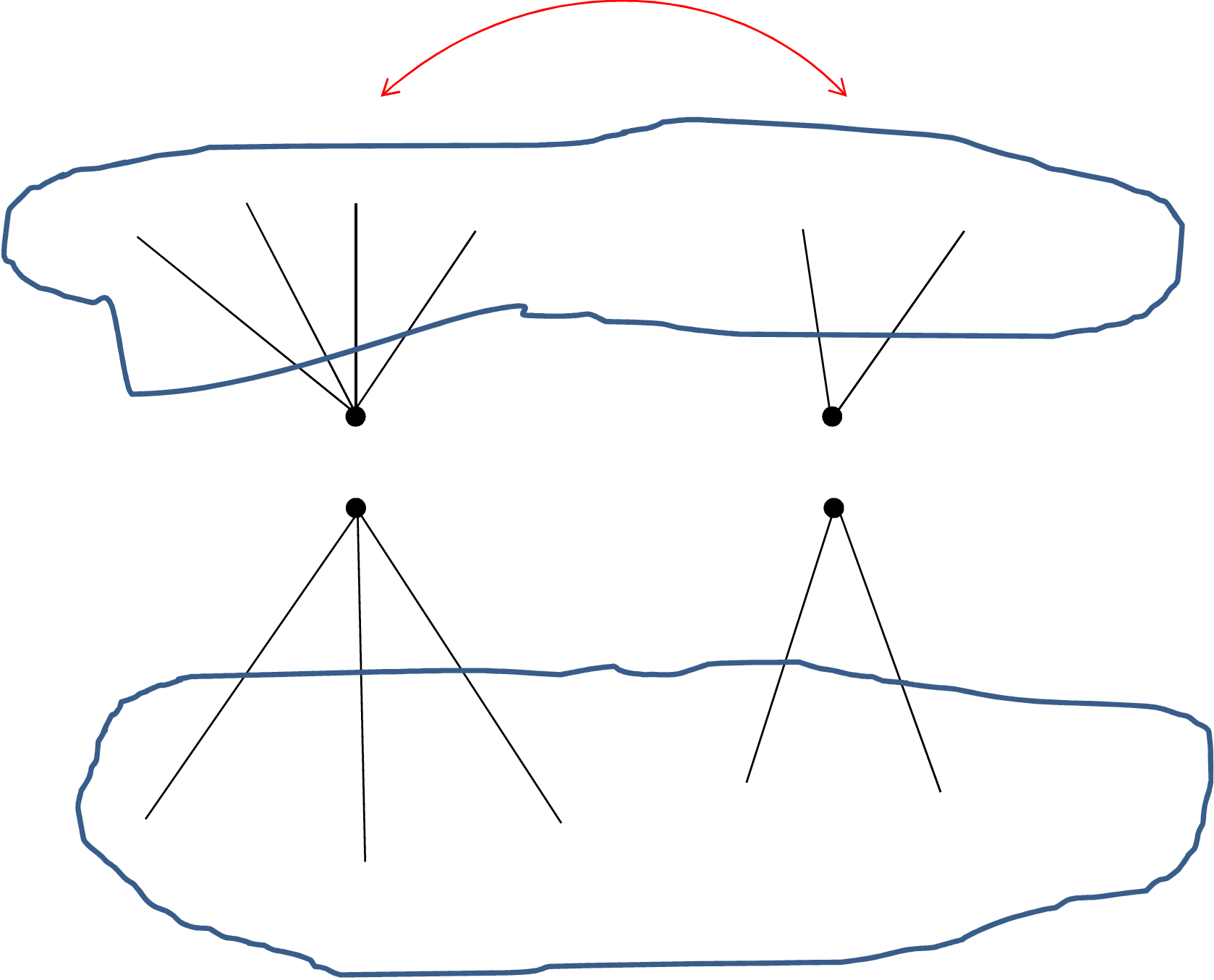}}&
  {  \includegraphics[scale=.16]{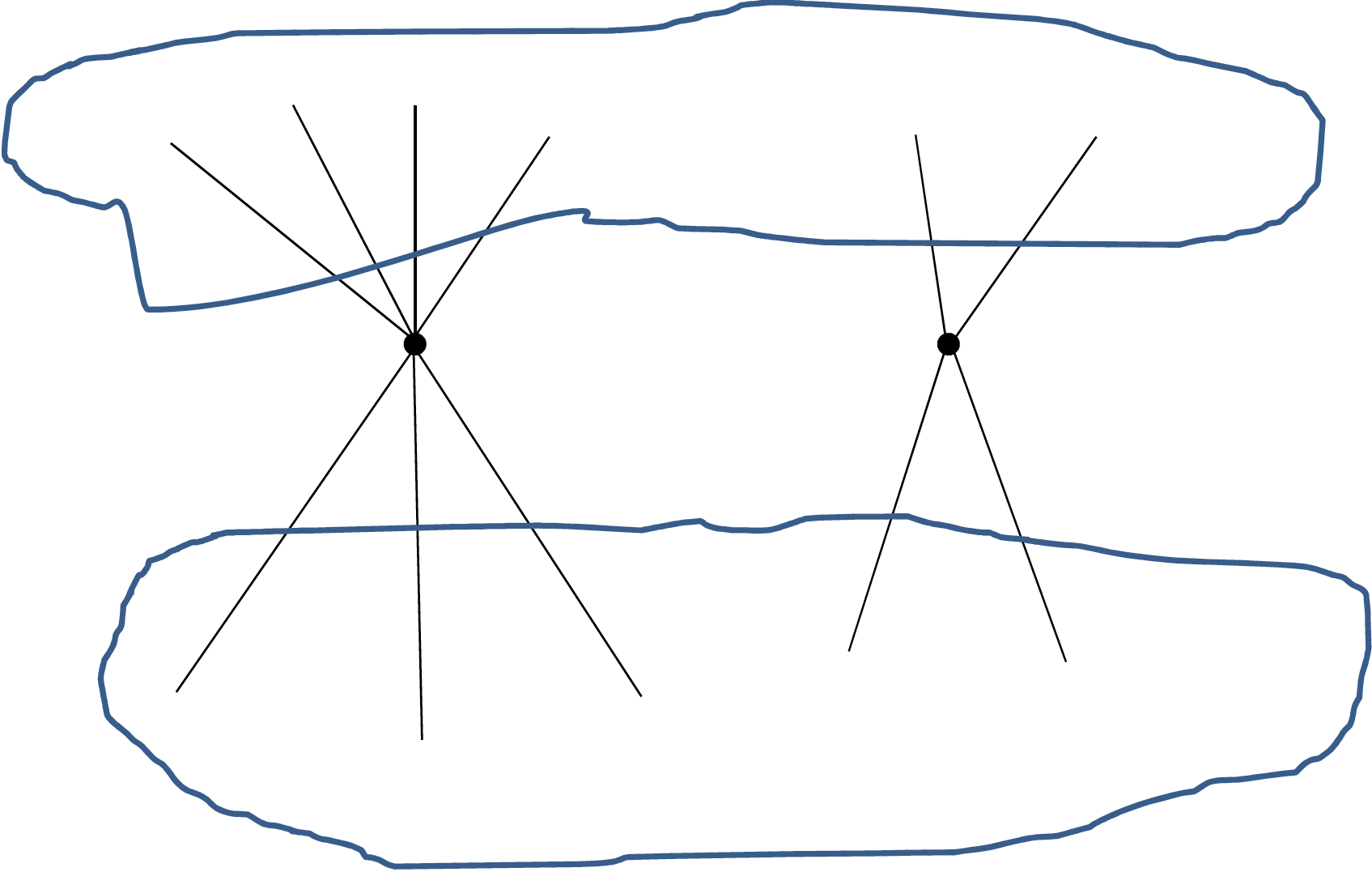}}
\\
  (a) & (b) &(c)&(d) \\
\end{tabular} \end{center}
  \caption{The reversal operation. The graphs, (a) and (d) are 2-isomorphic
  but not isomorphic.
Note that the edge lengths of the frameworks are unchanged under a
2-isomorphism.}
  \label{fig:reverse}
\end{figure}

A theorem of Whitney~\cite{W33}
(see also~\cite{sanders})
allows us to upgrade cycle isomorphisms
to graph isomorphisms.
\begin{theorem}
\label{thm:whit}
Let $G$ and $H$ be two graphs, with
$G$ being $3$-connected, and
$H$ with no isolated vertices.
An edge bijection that is a  cycle isomorphism 
must arise from a graph isomorphism.
\end{theorem}
\begin{remark}
Another way to state Whitney's theorem is that if $G$ and $H$ 
have isomorphic graphic matroids and the same number of vertices,
and $G$ is $3$-connected, then $G$ and $H$ are isomorphic 
as graphs.  In particular, topological information 
contained in the ordering of the edges on a cycle isn't 
part of the 
hypothesis, nor did we consider it in Lemma \ref{lem:cycle}.

The notion of cycle isomorphism is equivalent to having isomorphic 
graphic matroids, so it could also be formulated in terms of  
``forest isomorphisms''.
\end{remark}
\begin{remark}
If $G$ is not $3$-connected, then there are cycle isomorphisms
between $G$ and non-isomorphic $H$.  Whitney \cite{Whitney} showed
that these belong to a restricted class of ``$1$-isomorphisms'' and 
``$2$-isomorphisms.'' See Figure~\ref{fig:reverse} for an example of
a pair of $2$-isomorphic graphs.
\end{remark}

With this in hand, we can prove the main proposition
of this section.

\begin{proof}[Proof of Proposition~\ref{prop:iso}]
From Lemma~\ref{lem:cycle} the mapping taking the edges of
$G$ to $H$ must be a cycle isomorphism. Then from the assumed
$3$-connectivity and Theorem~\ref{thm:whit} this mapping must
arise from a vertex relabeling.
\end{proof}

The main  structural theorem of
this paper now follows.

\begin{proof}[Proof of Theorem~\ref{thm:iso}]
From Proposition~\ref{prop:induct} we can reduce the problem
from $d$ dimensions down to $1$. Since $G$ is generically
globally rigid in $d \ge 2$ dimensions, from Theorem~\ref{thm:hen}
it must be $3$-connected. 
Since the number of vertices
in both graphs is the same
and $M_{1,G}=M_{1,H}$, from
Lemma~\ref{lem:infrig},
$H$ cannot have isolated vertices.
The result then follows from
Proposition~\ref{prop:iso}.
\end{proof}

As proven in the end of Section~\ref{sec:mv}, this immediately
proves the main result of this paper, Theorem~\ref{thm:main}. \qqed

\section{Bonus result}
\label{sec:bonus}

There is an interesting variant of Theorem~\ref{thm:iso} that was
originally reported in the unpublished manuscript~\cite{miso}.
For this theorem we will replace the hypothesis that $G$ is generically globally rigid in
$d$ dimensions with the far weaker one that $G$ is $3$-connected. 
However, we require  not only equality 
of  measurement varieties, but  also equality 
of Euclidean measurement sets.

\begin{definition}
Let $d$ be some fixed dimension and $n$ a number of vertices.
Let $G:= \{E_1,\ldots, E_m\}$ be  an ordered graph.
The ordering on the edges of $G$
fixes  an association between each edge in $G$ 
and a coordinate axis of $\RR^{m}$. 

We denote by $M^{\EE}_{d,G}$ the image of 
of $m^{\EE}_G(\cdot)$ over all real $d$-dimensional configurations.
We call this the (squared) \defn{Euclidean measurement set}
of $G$ (in $d$ dimensions). This is a real semi-algebraic set,
defined over $\QQ$.
 \end{definition}

\begin{theorem}
\label{thm:euc}
Let  $d$ be fixed.
Let $G$, an ordered 
$3$-connected graph on $n$ vertices with $m$ edges,
and $H$ some 
ordered graph with no 
isolated vertices and 
with $m$ edges.
Suppose $M^\EE_{d,G} = M^\EE_{d,H}$.
Then
there is a vertex relabeling of $H$ such that
$G=H$.
\end{theorem}
\begin{remark}
This theorem does not rule out the possibility that 
there are two non-isomorphic graphs $G$ and $H$
such that $M^\EE_{d,G} \cap M^\EE_{d,H}$ contains
a standard-topology open set.  This would 
imply that  $M_{d,G}$ is equal to $M_{d,H}$
even though  $M^\EE_{d,G} \neq M^\EE_{d,H}$.
In this case, there could be some generic Euclidean 
measurements that are achievable from both graphs
and some generic Euclidean measurements that are 
achievable only in one graph.
\end{remark}
As a result, Theorem~\ref{thm:euc} does not help us to prove 
Theorem~\ref{thm:main}.

The rest of this section proves Theorem \ref{thm:euc}.
The steps are similar to the ones we used in 
Section \ref{sec:d=1}, but in the present setting, 
they work immediately in dimensions greater than one.

\begin{lemma}
\label{lem:ind2}
Let $E'$ be subset of the edges of $G$. 
$\pi_{S_{\overline{E'}}}(M^{\EE}_{d,G})$ equals the entire first octant 
if and only if 
the edges of $E'$ form a 
forest over the vertices of $G$.
\end{lemma}
\begin{proof}
If $E'$ is a forest, then we given any target measurements point
in the first octant of $\RR^{|E'|}$, we
can traverse the forest and sequentially place the vertices 
in $\RR^d$ to achieve this measurement.

Conversely, 
if $E'$ is not a forest, then it contains a cycle $C$ on $k$ edges,
for some $k$.
In this case, 
$\pi_{S_{\overline{E'}}}(M^{\EE}_{d,G})$
cannot be the entire first
octant of $\RR^k$ since there is no real framework (in any dimension) where all
but one of the edges of the cycle has zero length
\end{proof}

\begin{lemma}
\label{lem:cycle2}
Let $G$ and $H$ be  ordered graphs with $m$ edges.
Suppose $M^{\EE}_{d,G} =M^{\EE}_{d,H}$.
Then the mapping 
taking the ordered edges of $G$ to the ordered edges of $H$
is a cycle isomorphism.
\end{lemma}
\begin{proof}
Suppose the mapping  is not a cycle isomorphism. Then wlog
there is
a set of edges  $C$ that from
a simple cycle in $G$ and not $H$.

Suppose that $C$ is a forest in $H$.
Then from Lemma~\ref{lem:ind2}, we have
$\pi_{S_{\overline{C}}}(M^{\EE}_{d,G})
\neq 
\pi_{S_{\overline{C}}}(M^{\EE}_{d,H})$ and thus
$M^{\EE}_{d,G} \neq M^{\EE}_{d,H}$.

Suppose that $C$ is neither a simple cycle nor a forest in $H$,
then there must be an edge $e$ such that $C':=C-e$ is not a forest
in $H$, while $C'$ is a forest in $G$.
Then from Lemma~\ref{lem:ind2}, we have
$\pi_{S_{\overline{C'}}}(M^{\EE}_{d,G})
\neq \pi_{S_{\overline{C'}}}(M^{\EE}_{d,H})$ and thus
$M^{\EE}_{d,G} \neq M^{\EE}_{d,H}$.
\end{proof}

\begin{proof}[Proof of Theorem~\ref{thm:euc}]
The theorem now follows directly from Lemma~\ref{lem:cycle2}, 
the assumed $3$-connectivity, and Theorem~\ref{thm:whit}.
\end{proof}

\section{Remaining Issues}\label{sec:issues}

This paper answers some central questions about the relationships
between graphs and their measurement varieties/sets. There are a few natural remaining questions.

\subsection{Redundant Rigidity}
\label{sec:rr}
Theorem~\ref{thm:main} is tight in the sense that if $G$ is
not generically globally rigid  in $d$ dimensions, then clearly we cannot determine $\p$ from
$\v$. But it is still possible that one might be able to 
determine $G$ from $\v$.
Here we discuss a possible
strengthening of Theorem~\ref{thm:iso}.

\begin{definition}
We say that a graph $G$ is \defn{generically redundantly rigid}
in $\RR^d$ if $G$ is generically locally rigid in $\RR^d$ and remains
so after the removal of any single edge.
\end{definition}

\begin{question}
\label{qu:rr}
Is the following claim true:

Let $G$, an ordered graph with $n \ge d+2$ vertices and  $m$ edges,
be $3$-connected and 
generically redundantly rigid in $d$ dimensions.
Let  $H$ be some ordered graph on $n$ vertices
with $m$ edges.
Suppose $M_{d,G} =M_{d,H}$.
Then
there is a vertex relabeling under which $G=H$.
\end{question}

The claim is true
for $d = 2$, since  in two dimensions,
redundant rigidity and $3$-connectivity imply
generic global  rigidity~\cite{conGR,jj}.

\begin{figure}[htbp]
	\centering
	\subfloat[]{\includegraphics[width=0.340\textwidth]{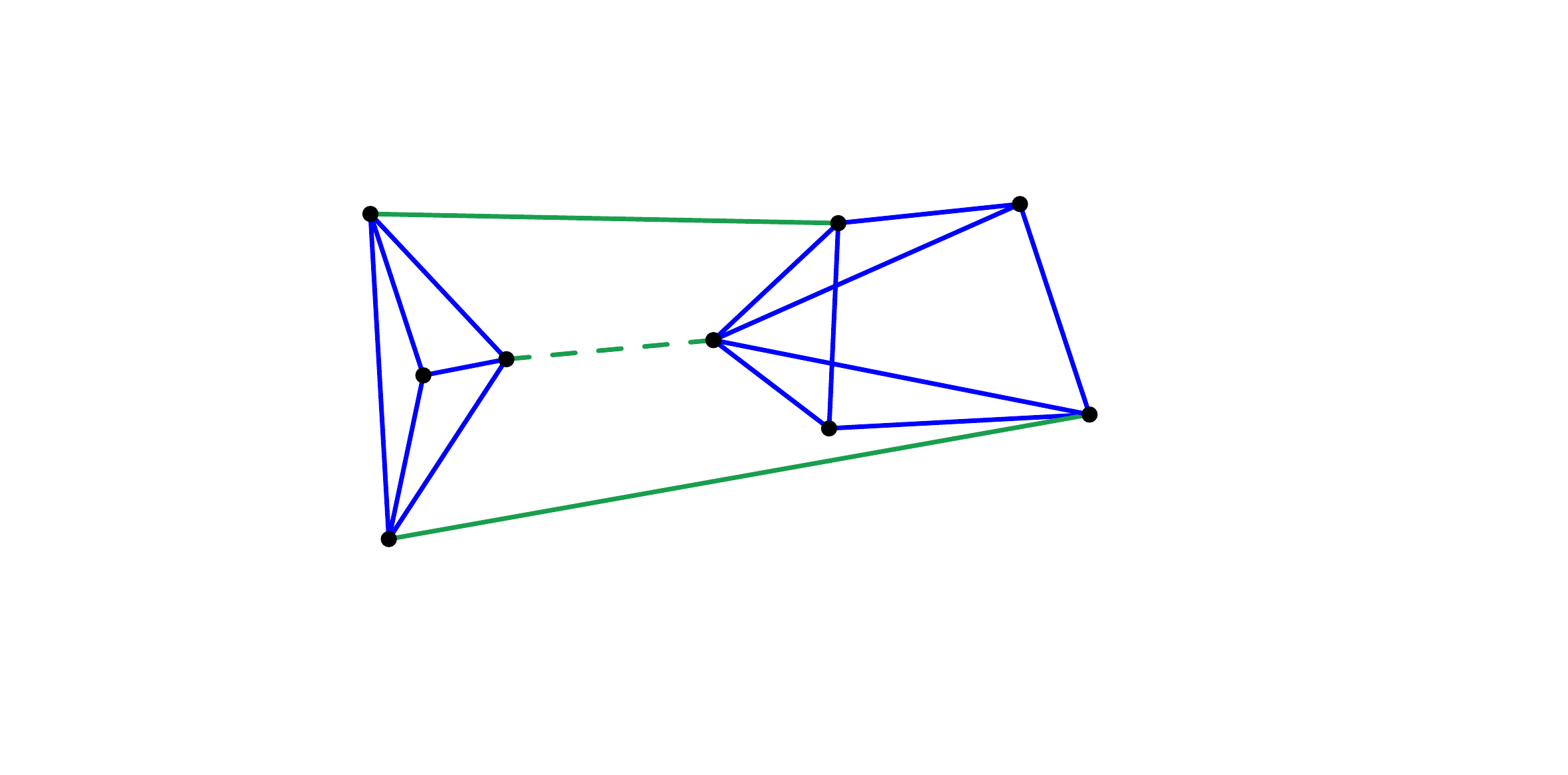}}
    \qquad\qquad
    \subfloat[]{\includegraphics[width=0.340\textwidth]{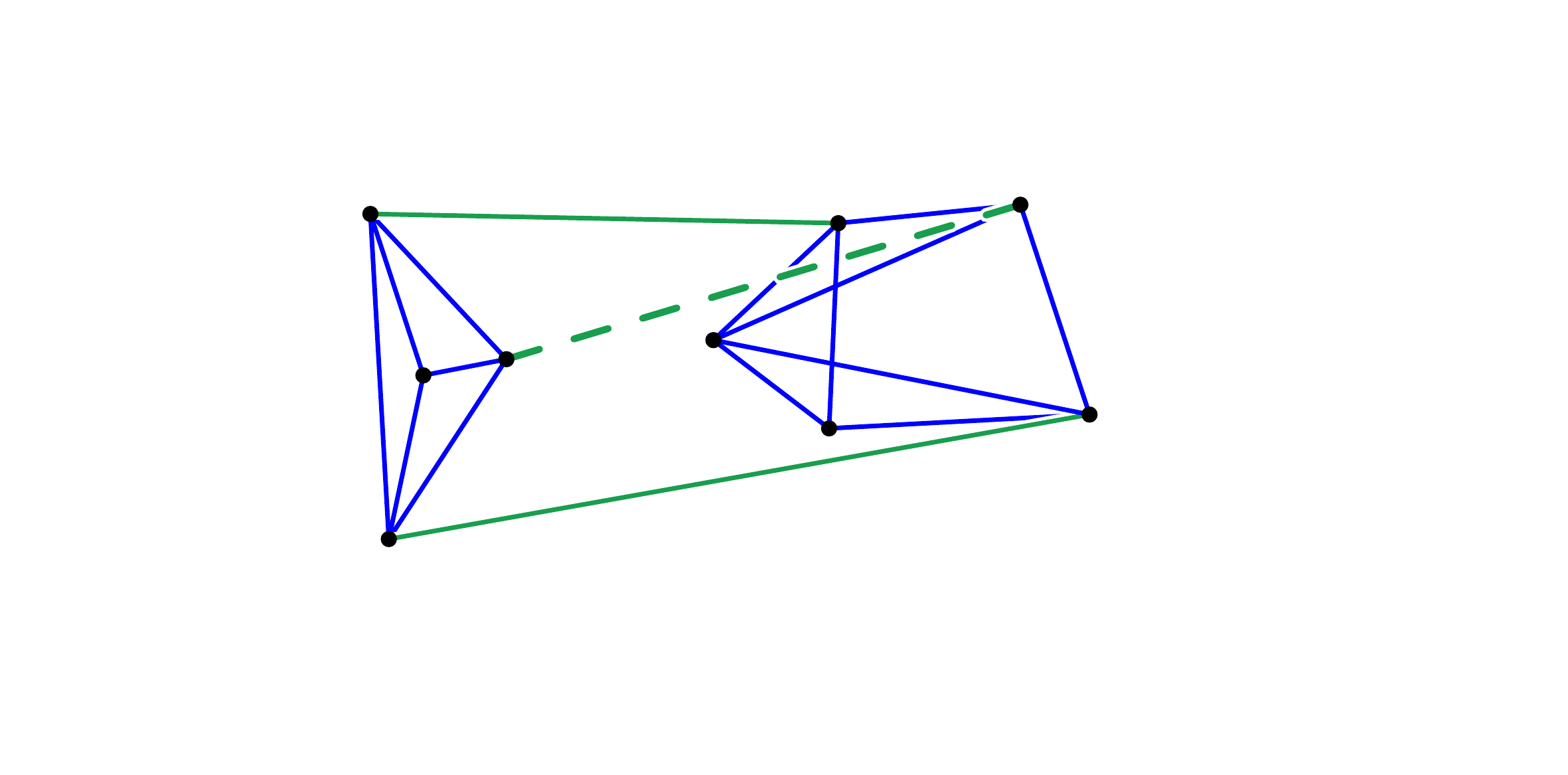}}
    
\caption{A pair of non-isomorphic graphs with the same 
measurement variety.
Any of the green edges in the graphs (a) and (b), if 
removed, result in a graph that is generically 
flexible.  As described in Remark \ref{rem:not-grr}, these 
graphs have the same measurement variety (it is the product 
of the measurement varieties of the complete graph $K_4$,
the wheel $W_4$ and $\CC^3$).  
However, the graphs (a) and (b) are not isomorphic, because the dashed green edges 
can be distinguished from each other by the degree 
of the endpoint on the right.}
\label{fig:not-rr}
\end{figure}

In terms of the the ingredients used for proving
Theorem~\ref{thm:iso}, we note that the conclusion of
Proposition~\ref{prop:mainA} is false when
$G$ is merely generically redundantly rigid.
For example
the complete bipartite graph, $K_{5,5}$, is redundantly rigid
in $3$ dimensions and is $4$-connected. But for any configuration
$\q$ where 
even one of its ``parts'' has a deficient span, there
will be an infinite number of  generic Gauss fibers
with $m_G(\q)$ it their closure. 
This is because  
equilibrium stresses for generic $\p$, which are all
rank $2$, only enforce   
affine relations within each of the parts~\cite{br}.

A positive answer to Question~\ref{qu:rr}
would
directly imply Theorem~\ref{thm:iso} due to
Theorem~\ref{thm:hen} and 
the following theorem of Hendrickson~\cite{hen}.
\begin{theorem}
\label{thm:hen2}
If $G$ is generically globally rigid in $\RR^d$, 
with $n\ge d+2$
then it is redundantly
rigid in $\RR^d$.
\end{theorem}

\begin{remark}\label{rem:not-grr}
A positive answer to Question~\ref{qu:rr} would give us a reasonably 
tight characterization
of measurement variety agreement in light of the following.

Suppose that $G$ is not generically redundantly rigid, and let
$e$ be an an edge of $G$ so that  $G' := G - e$ is generically locally flexible. From the size of $G$, there must be 
a non-edge $e'$ of $G$ different from $e$
whose lengths can be changed under a continuous  flex of $G'$.
Let $G''$ be the graph obtained from $G$ by replacing $e$ with $e'$.
Then $M_{d,G} = M_{d,G''}$ as both are equal to $M_{d,G'} \oplus \CC^1$.
(See an example in Figure~\ref{fig:not-rr}.)
But the mapping from $G$ to $G''$ will not be a an isomorphism for graphs unless $e$ and $e'$ are in the same orbit of 
$\operatorname{Aut}(G + e')$.  This is a very restrictive 
condition on $G$ that any extension of our results to graphs that are not generically 
redundantly rigid will have to include%
\footnote{Garamvölgyi and Jordán \cite{tibor} explore the question of when a non-redundantly rigid graph can be
reconstructed from edge-length measurements.}%
.
\end{remark}

There are also some more unresolved issues about
measurement sets.

\begin{question}
Can the assumption that $G$ and $H$ have the same
number of vertices be dropped from Theorem~\ref{thm:iso}?
(This open up the possibility that
$H$ has more vertices, but is 
generically locally flexible.)\footnote{Garamvölgyi and Jordán \cite{tibor} give an affirmative answer in dimensions one and two.}
\end{question}

More generally, we know very little about 
what assumptions other than dimension
can let us conclude for general $d$ that 
$M_{d,G} \subseteq M_{d,H}$
implies $M_{d,G}=M_{d,H}$.

\subsection{Unlabeled graph realization}\label{sec: practical}
The \textit{graph realization} (or \textit{distance geometry})
problem asks to reconstruct an unknown configuration $\p$ given 
a graph $G$, dimension $d$, and labeled edge length 
measurements $\v = m^{\EE}_G(\p)$.  From 
Theorem \ref{thm:ght}, if we assume that $\p$ is 
generic and know that $G$ is generically globally
rigid, and we can find any $\q$ at all (not necessarily 
generic) so that $\v = m^\EE_G(\q)$, then we know that
$\q = \p$ (up to congruence).  As a practical matter, 
it is important that $\q$ need not be generic, since 
this is a very strong restriction on (or assumption about)
any specific algorithm (as opposed to the process generating
the input $\p$).

Theorem \ref{thm:main} doesn't immediately 
give us the analogous result for \textit{unlabeled}
distance geometry. The subtlety is that the hypotheses 
of Theorem \ref{thm:main} include knowledge about $G$,
which we won't have access to in an unlabeled distance 
geometry instance.  Unlike in the labelled case, 
$\v$ by itself doesn't immediate tell us whether our
problem is generically well-posed.

Examining our proofs, we get a partial result in 
this direction:
\begin{theorem}\label{thm: practical}
In any fixed dimension $d \ge 2$, 
let $\p$ be a generic configuration of $n \ge d+2$ 
points. Let
$\v= m^{\EE}_G(\p)$, where 
$G$ 
is an ordered graph 
(with $n$ vertices and $m$ edges)
that is generically 
locally rigid in $\RR^d$.

Suppose there is a 
configuration $\q$, 
also of $n$ points,
along with 
an ordered graph $H$
(with $n$ vertices and $m$ edges) that
is generically globally rigid and 
such that 
$\v= m^{\EE}_H(\q)$.

Then
there is a vertex relabeling of $H$ such that
$G=H$.
Moreover, under this vertex relabeling,
up to congruence,
$\q=\p$.
\end{theorem}
\begin{proof}[Proof sketch]
The derivation of Theorem \ref{thm:main}
from Theorem \ref{thm:iso} works nearly 
unmodified.  Both $G$ and $H$ are generically
locally rigid by hypothesis
and so 
$M_{d,G}$
and $M_{d,H}$ are of the same dimension.
The configuration $\p$ maps to a 
generic point in the intersection of $M_{d,G}$
and $M_{d,H}$, so the two measurement varieties 
are equal. 
When applying~\ref{thm:iso}, we rely on the generic global rigidity of $H$ instead of $G$.
The rest of the proof then goes through unchanged.
\end{proof}
In this version, we only need to assume that $G$ is
generically locally rigid, instead of generically globally rigid.
Theorem \ref{thm: practical}
then tells us that $H$ (whose generic global
rigidity can be tested in a realization setting)
and $\q$ certify that
the input problem is, in fact, well-posed, and 
that we have found its solution.
This version still makes some assumptions on $G$
and the number of vertices in $H$.

This motivates the following question.




\begin{question}
Is the following claim true:

In any fixed dimension $d \ge 2$, 
let $\p$ be a generic configuration of $n \ge d+2$ 
points. Let
$\v= m^{\EE}_G(\p)$, where 
$G$ 
is an ordered graph 
(with $n$ vertices and $m$ edges).
Let $\p_S$ be the subconfiguration
of $\p$ indexed by the vertices
within the support of $G$.

Suppose there is a 
configuration $\q$, 
of $n'$ points, with no two points coincident,
along with 
an ordered generically globally rigid graph $H$
(with $n'$ vertices and $m'$ edges)
such that 
$\v= m^{\EE}_H(\q)$.

Then 
there is a vertex relabeling  of $\p_S$ such that,
up to congruence,
$\q=\p_S$.
Moreover, under this vertex relabeling,
$G=H$.
\end{question}
A positive answer to this question would mean that,
under the assumption that $\p$ is generic, 
such an $(H,\q)$ would be a certificate that
we have correctly realized the measured 
subconfiguration of $\p$. 

The difficulty for this question is that we do not know how to rule out
the possibility that
$M_{d,G}\subsetneq M_{d,H}$.
This is related to the issues mentioned at the 
end of Section~\ref{sec:rr}.

There is one special case of note. 
The claim of this question is true when 
$H$ is the complete graph $K_{d+2}$. 
(see~\cite[Proposition 4.23]{loops}).
This is due to the fact that,
aside from $K_{d+2}$, any other graph $G$ with
$N:=\binom{d+1}{2}$ distinct
edges has the property that every subset of 
edges is independent.  Hence, 
the measurement variety, $M_{d,G}$, of $G$ must be equal to all of $\CC^N$, 
and so it cannot be a subset of $M_{d,H}$. 
Applying this idea iteratively, it can be
shown that the claim of the question remains true
if $H$ allows for trilateration~\cite{loops}.
This fact allows one to apply trilateration to an unlablled set of measurements, 
as is done in~\cite{dux1}, without any assumptions on  $G$ or $n$.

\subsection{Matrix completion}
A variant of global rigidity is ``matrix completion'',
which asks whether all the entries of an $m\times n$ 
matrix $A$ of (low) rank $r$  can be determined 
by a subset of its entries (at known positions).  
(See \cite{SC09} for complete definitions and background.)

The algebraic setup (see \cite{KTT15})
takes $A$ as a point on the 
determinantal variety of $m\times n$ matrices of 
rank at most $r$, and the observation process is the 
projection onto coordinates corresponding to the 
entries.  The closure of the image of this projection corresponds 
to the measurement variety of a framework.  For complex
matrix completion, a result of \cite{KTT15} says that
whether an observation pattern has a unique completion 
is a generic property.  This means it makes sense to 
ask whether our results also hold in the matrix completion 
setting.

The following rank $3$ examples are from \cite{KTT15v3} (a
preprint version of \cite{KTT15}).
\[
\begin{pmatrix}
\star & \star & \star & \star & \star \\
\star & \star & \star & \star & \star \\
\star & \star & \star & \star & \star \\
\star & \star & \star & ?     & ? \\
\star & \star & \star & ?     & ?  \\
\star & \star & \star & ?     & ?  
\end{pmatrix}
\qquad
\begin{pmatrix}
\star & \star & \star & \star & ? \\
\star & \star & \star & \star & ? \\
\star & \star & ? & \star & \star \\
\star & \star & ? & \star & \star  \\
? & \star & \star & \star & \star  \\
? & \star & \star & \star & \star
\end{pmatrix}
\]
It is shown there that that, for each of these, 
the projection onto the known entries (labeled ``$\star$'')
is dominant.  Hence, they both have the same ``measurement 
varieties''.  Additionally, if the underlying matrix is 
generic, there is exactly one way to fill in the 
unknown entries (labeled ``?''), so they are also 
``globally rigid''.

Importantly, they are \textit{not} related by 
row and column permutations, so we have a 
counter-example to the straightforward translation 
of our main results to the matrix completion setting.
(What goes wrong is that the stress 
criterion for global rigidity isn't necessary 
for matrix completion.  This was first observed in \cite{SC09}.)

On the other hand, as noted in 
\cite[Remark 4.20]{loops}, the matrix
completion analogue of Boutin and Kemper's 
result for complete graphs is straightforward.
Clarifying the relationship between unlabeled
matrix completion and unlabeled rigidity would
be interesting.

\newpage
\appendix

\section{Algebraic Geometry Background}

\begin{definition}
A (complex embedded affine) \defn{variety} 
(or \defn{algebraic set}), $V$,
is a (not necessarily strict)
subset of $\CC^N$, for some $N$,
that is defined by the simultaneous
vanishing of a finite set of polynomial equations 
with coefficients in $\CC$
in the 
variables $x_1, x_2, \ldots, x_N$ which are associated with the 
coordinate axes of $\CC^N$.
We say that $V$ is \defn{defined over} $\QQ$ if it can 
be defined by polynomials with coefficients in $\QQ$.

A variety is homogeneous if its ideal is 
finitely generated by homogeneous polynomials. 
This is the same as the set $V$ being 
a cone with its vertex at $0$.

A variety can be stratified as a union of a finite number of
complex analytic submanifolds of $\CC^N$.
A variety $V$ has a well defined (maximal) \defn{dimension} $\Dim(V)$, 
which will agree with the largest $D$ for which there
is a standard-topology
open subset of~$V$,  
that is a $D$-dimensional complex analytic submanifold of $\CC^N$.

The set of polynomials that vanish on $V$ form 
a radical ideal $I(V)$, which is generated by a finite set
of polynomials.

A variety $V$ is \defn{reducible} if it is the proper union of two
varieties $V_1$ and $V_2$. 
Otherwise it is called
\defn{irreducible}.
A variety has a unique decomposition as a finite proper
union of its
maximal irreducible subvarieties called \defn{components}.

Any (strict) subvariety $W$ of an 
irreducible variety $V$ must be of strictly lower dimension.

A subset $W$ of a variety $V$ is called \defn{Zariski closed}  if $W$ is
a variety.

\end{definition}

\begin{definition}
The \defn{Zariski tangent space} at a point $\x$ of a
variety $V$
is the kernel of the Jacobian matrix of a set of
generating polynomials for $I(V)$ evaluated at $\x$.

A point $\x$ of an irreducible variety $V$
is called
(algebraically) \defn{smooth} in $V$ if the dimension of the Zariski
tangent space equals the  dimension of $V$.
Otherwise $\x$ is called
(algebraically) \defn{singular} in $V$.

A smooth point $\x$ in an irreducible variety
$V$ 
has a standard-topology neighborhood
in $V$ that is a complex analytic submanifold of $\CC^N$ 
of dimension $\Dim(V)$.

The \defn{locus} of singular points of $V$ is denoted $\sing(V)$.
The singular locus is itself a strict subvariety of $V$. 
\end{definition}

\begin{definition}
A \defn{constructible set} $S$ is a set that can be defined using a finite
number of varieties and a finite number of Boolean set operations.
We say that $S$ is \defn{defined over} $\QQ$ if it can 
be defined by polynomials with coefficients in $\QQ$.

$S$ has a well defined (maximal) \defn{dimension} $\Dim(S)$, 
which will agree with the largest $S$ for which there
is a standard-topology 
open subset of~$S$, 
that is a $D$-dimensional complex analytic submanifold of $\CC^N$.

The \defn{Zariski closure} of $S$ is the smallest variety $V$ containing it.
The set $S$ has the same dimension as its Zariski closure $V$.
If $S$ is defined over $\QQ$, then so too is $V$ 
(this can be shown using the fact that $S$ is invariant to 
elements of the absolute Galois group of $\QQ$).

The image of a variety $V$ 
under a polynomial map is a constructible set $S$.
If $V$ is defined over $\QQ$, then so too is 
$S$~\cite[Theorem 1.22]{basu}.
If $V$ is irreducible, then so 
too is the Zariski closure of $S$. (We say that $S$
is \defn{irreducible}.)

\end{definition}

The following can be found in~\cite[Prop 10.1]{milne}.
\begin{lemma}
\label{lem:milne}
An irreducible constructible set $S$ contains a Zariski open subset
of its Zariski closure $V$. Thus $V\setminus S$ is 
contained in a subvariety $W$ of $V$. If $S$ is defined over $\QQ$,
there is such a $W$ that is as well.
\end{lemma}

\begin{definition}
A point in an irreducible variety 
or constructible set,
$V$ defined over $\QQ$ is called 
\defn{generic} if its coordinates do not satisfy any algebraic equation 
with coefficients in $\QQ$ besides those that are satisfied by every point
in $V$.  

The set of generic points has full measure 
in $V$.

When $V$ is an irreducible variety
and defined over $\QQ$, 
all of its generic
points are smooth.

A generic real configuration in $\RR^d$ 
(as in Definition~\ref{def:genConig})
is also a generic point in $\CC^N$,
considered as a variety, as in the current definition.
\end{definition}

\begin{lemma}
\label{lem:dense}
Let $V \subseteq W$ be an inclusion of varieties where $W$ 
and $V$ are 
irreducible and $W$ is 
defined over $\QQ$. Suppose that $V$ 
has at least
one point $\y$ which is generic in $W$ (over $\QQ$). Then the points in
$V$ which are generic in $W$ are 
Zariski dense in $V$.
\end{lemma}

\begin{proof}
  Let~$\phi$ be a non-zero algebraic function on~$W$ defined over $\QQ$.
  Consider the Zariski open subset
set $X_\phi :=
  \{\,\x \in V \mid \phi(\x) \ne 0\,\}$.  
This is non empty due to our assumption about the point 
$\y$.
Thus, from the irreducibility of $V$, 
this $X_\phi$ is Zariski dense in 
$V$. 

The set of points in
$V$ which are generic in $W$ is defined as the intersection of these open and dense 
$X_\phi$ as $\phi$ ranges over the countable set of possible $\phi$.

When $U$ is any Zariski open and dense subset of $V$,
then $V\setminus U$ is contained in a strict subvariety of $V$.
From irreducibility and dimension considerations then, $U$  must
contain a standard-topology open and dense subset of
the smooth locus of $V$~\footnote{
In fact, using ~\cite[Theorem 1, Page 58]{mumford},
we can see that $U$ is
standard-topology open and dense in all of $V$.}.

As the smooth locus of $V$ under the standard topology is a Baire space,
a countable intersection of such subsets is 
standard-topology dense in the smooth locus of $V$.
Thus, again from irreducibility and 
dimension considerations, this intersection is Zariski dense in all of $V$.
\end{proof}

\begin{lemma}
\label{lem:genMap}
Let $V$ be an irreducible variety and $f$ a 
polynomial map, 
both 
defined over
$\QQ$. 
Then the image of a generic point in $V$ is generic
in $f(V)$.
\end{lemma}

\begin{lemma}
\label{lem:preG}
Let $V$ be a
 irreducible variety, $f$ be a polynomial map $f:V \rightarrow \CC^m$ 
all defined over $\QQ$.
Let $W := f(V)$.
If $\y$  is generic in $W$, there is a point in $f^{-1}(\y)$ that is generic in V.
\end{lemma}
\begin{proof}
  Let~$\phi$ be a non-zero algebraic function on~$V$ defined over $\QQ$.
  We start by showing there is a point $\x \in f^{-1}(\y)$ so that
  $\phi(\x) \ne 0$.  
Consider the constructible set $X_\phi :=
  \{\,\x \in V \mid \phi(\x) \ne 0\,\}$.  
  This is Zariski dense in~$V$ due
to irreducibility,
so its
  image $f(X_\phi)$ is Zariksi dense in~$W\!$. 
  Therefore, from Lemma~\ref{lem:milne},
  $Y_\phi := W
  \setminus f(X_\phi)$ 
  is contained in  some proper subvariety $T$ of $W$ defined
  over $\QQ$.
  
  But then since $\y$ is generic 
  it cannot be in $T$, so 
  $\y$ is in the image of $X_\phi$, so there is an $\x \in
  f^{-1}(\y)$ such that $\phi(\x) \ne 0$, as desired.

  Let $Z_\phi = \{\,\x\in f^{-1}(\y) \mid \phi(\x) = 0\,\}$.  We have
  shown $Z_\phi$ is a proper subset of $f^{-1}(\y)$ for any non-zero
  algebraic function~$\phi$ on $V$,
  defined over $\QQ$. It follows that for any finite
  collection of $\phi_i$, the union of the $Z_{\phi_i}$ is still a
  proper subset of $f^{-1}(\y)$ (as we can consider the
  product of the $\phi_i$).  But there are only countably many
  possible $\phi$ overall, and a countable union of algebraic subsets
  covers an algebraic set iff some finite collection of them do.
  (Proof: this is true for each irreducible component, as a proper
  algebraic subset has measure zero, and there are only finitely many
  irreducible components.)  Thus
  the union of the $Z_\phi$ do not cover $f^{-1}(\y)$, i.e., there is
  a generic point in $f^{-1}(\y)$.
\end{proof}

\newpage

\end{document}